%% file: ms.tex
\documentclass[hidelinks]{siamart220329}

\usepackage{algpseudocode}
\usepackage{amsfonts} 
\usepackage{amsopn}
\usepackage{amssymb}
\usepackage{bm}
\usepackage{booktabs}
\usepackage{caption}
\usepackage{commath}
\usepackage{csvsimple}
\usepackage{enumitem}
\usepackage[T1]{fontenc}
\usepackage{mathtools}
\usepackage{subcaption}
\usepackage{tikz}
\usepackage{lipsum}
\usepackage{slantsc}

\newsiamremark{assumption}{Assumption}
\newsiamremark{example}{Example}
\newsiamremark{remark}{Remark}
\crefname{assumption}{Assumption}{Assumptions}
\crefname{remark}{Remark}{Remarks}

\def\NUMA{%
    NUMA, Department of Computer Science, KU Leuven, Leuven, Belgium
}
\def\corresponding{%
    Corresponding author: \email{arne.bouillon@kuleuven.be}%
}

\makeatletter
\renewcommand*\env@matrix[1][*\c@MaxMatrixCols c]{%
  \hskip -\arraycolsep
  \let\@ifnextchar\new@ifnextchar
  \array{#1}}
\makeatother

\newenvironment{smallarray}[1]
 {\null\,\vcenter\bgroup\scriptsize
  \arraycolsep=.13885em
  \hbox\bgroup$\array{@{}#1@{}}}
 {\endarray$\egroup\egroup\,\null}

\newcommand\kronm{}

\newcommand\kronv\mv

\newcommand\ad\lambda
\newcommand\Ad\Lambda

\newcommand\kron\otimes

\newcommand\trsp{{\mathstrut\scriptscriptstyle{\top}}}

\newcommand\iu{{\mathrm{i}\mkern1mu}}

\newcommand\E{{\mathrm e}}

\newcommand\ex{{\mathrm{ex}}}

\def\mper{{.}}
\def\mcom{{,}}

\def\Sidx{M}

\def\tidx{l}
\def\Tidx{L}

\def\iidx{k}

\newcommand*\dualp[2]{{\langle#1,#2\rangle}}

\newcommand\dt{{\delta t}}

\newcommand\Dt{{\Delta t}}

\newcommand\DT{{\Delta T}}

\newcommand\ukus[2]{{#2}}
\def\sz{\ukus{s}{z}}

\newcommand*\mv[1]{{\bm{{#1}}}}

\DeclareMathOperator\diag{diag}

\definecolor{matlabred}{rgb}{0.6350,0.0780,0.1840}
\definecolor{matlabgreen}{rgb}{0.4660,0.6740,0.1880}
\definecolor{matlabblue}{rgb}{0,0.4470,0.7410}

\algnewcommand{\LineComment}[1]{\Statex \(\triangleright\) #1}
\algnewcommand{\LineCommentWithSkip}[1]{\Statex \hskip 1cm \(\triangleright\) #1}

\usepackage{letltxmacro}
\LetLtxMacro\orgvdots\vdots
\LetLtxMacro\orgddots\ddots

\makeatletter
\DeclareRobustCommand\vdots{%
  \mathpalette\@vdots{}%
}
\newcommand*{\@vdots}[2]{%
  \sbox0{$#1\cdotp\cdotp\cdotp\m@th$}%
  \sbox2{$#1.\m@th$}%
  \vbox{%
    \dimen@=\wd0 %
    \advance\dimen@ -3\ht2 %
    \kern.5\dimen@
    \dimen@=\wd2 %
    \advance\dimen@ -\ht2 %
    \dimen2=\wd0 %
    \advance\dimen2 -\dimen@
    \vbox to \dimen2{%
      \offinterlineskip
      \copy2 \vfill\copy2 \vfill\copy2 %
    }%
  }%
}
\DeclareRobustCommand\ddots{%
  \mathinner{%
    \mathpalette\@ddots{}%
    \mkern\thinmuskip
  }%
}
\newcommand*{\@ddots}[2]{%
  \sbox0{$#1\cdotp\cdotp\cdotp\m@th$}%
  \sbox2{$#1.\m@th$}%
  \vbox{%
    \dimen@=\wd0 %
    \advance\dimen@ -3\ht2 %
    \kern.5\dimen@
    \dimen@=\wd2 %
    \advance\dimen@ -\ht2 %
    \dimen2=\wd0 %
    \advance\dimen2 -\dimen@
    \vbox to \dimen2{%
      \offinterlineskip
      \hbox{$#1\mathpunct{.}\m@th$}%
      \vfill
      \hbox{$#1\mathpunct{\kern\wd2}\mathpunct{.}\m@th$}%
      \vfill
      \hbox{$#1\mathpunct{\kern\wd2}\mathpunct{\kern\wd2}\mathpunct{.}\m@th$}%
    }%
  }%
}
\makeatother

\headers{Diagonali\sz{}ation-based ParaOpt preconditioners}{A.\ Bouillon, G.\ Samaey, and K.\ Meerbergen}
\title{Diagonali\ukus{s}{z}ation-based preconditioners and generali\ukus{s}{z}ed convergence bounds for ParaOpt}
\author{{Arne Bouillon\thanks{\NUMA}} \thanks{\corresponding}\and Giovanni Samaey\footnotemark[1]\and Karl Meerbergen\footnotemark[1]}

\begin{document}

\maketitle

\begin{abstract}
    \input{src/abstract}
\end{abstract}

\begin{keywords}
    \input{src/keywords}
\end{keywords}

\begin{MSCcodes}
    \input{src/msc-codes}
\end{MSCcodes}

\section{Introduction} \label{sec:intro}
    \input{src/intro/intro}

\section{The ParaOpt algorithm} \label{sec:paraopt}
    \input{src/paraopt/paraopt}

\section{ParaOpt convergence} \label{sec:conv}
    \input{src/conv/intro}

    \subsection{The linear diffusive setting} \label{sec:conv:setting}
    \input{src/conv/setting}

    \subsection{Convergence results} \label{sec:conv:conv}
    \input{src/conv/conv}

    \subsection{Interpreting the convergence results} \label{sec:conv:interp}
    \input{src/conv/interp}

\clearpage
\section{Diagonali\sz{}ation-based preconditioners} \label{sec:diag}
    \input{src/diag/intro}

    \subsection{Scaling ParaOpt and the need for preconditioners} \label{sec:diag:scale}
    \input{src/diag/scale}

    \subsection{Linear coarse-grid correction} \label{sec:diag:linear}
    \input{src/diag/linear}

    \subsection{Formulating the preconditioners} \label{sec:diag:prec}
    \input{src/diag/prec}

    \subsection{Solving the smaller systems} \label{sec:diag:small}
    \input{src/diag/small}

    \subsection{Convergence results} \label{sec:diag:conv}
    \input{src/diag/conv}

\section{Proof of \cref{thm:conv:conv:tr-gen}} \label{sec:proof-tr}
    \input{src/proof-tr/proof-tr}

\section{Proof of \cref{thm:conv:conv:tc-gen}} \label{sec:proof-tc}
    \input{src/proof-tc/proof-tc}

\section{Numerical results} \label{sec:num}
    \input{src/num/intro}

    \subsection{Assessing the bounds for tracking} \label{sec:num:conv-tr}
    \input{src/num/conv-tr}

    \subsection{Assessing the bounds for terminal cost} \label{sec:num:conv-tc}
    \input{src/num/conv-tc}

    \subsection{Preconditioning} \label{sec:num:prec}
    \input{src/num/prec}

\section{Conclusions} \label{sec:concl}
    \input{src/concl/concl}

\clearpage
\appendix
\section{Properties of some ParaOpt propagators} \label{sec:apdx-po-prop}
    \input{src/apdx-po-prop/intro}

    \subsection{On $\varphi$ and $\psi$, and how to find them} \label{sec:apdx-po-prop:phipsi}
    \input{src/apdx-po-prop/phipsi}

    \subsection{Proof of \cref{lmm:conv:setting:prop-tr-ie}} \label{sec:apdx-po-prop:prop-tr-ie}
    \input{src/apdx-po-prop/prop-tr-ie}

    \subsection{Proof of \cref{lmm:conv:setting:prop-tr-ex}} \label{sec:apdx-po-prop:prop-tr-ex}
    \input{src/apdx-po-prop/prop-tr-ex}

    \subsection{Proof for \cref{ex:conv:conv:special}} \label{sec:apdx-po-prop:special}
    \input{src/apdx-po-prop/special}

\clearpage
\section*{Acknowledgments}
    \input{src/ack}

\bibliographystyle{siamplain}
\bibliography{references}
    
\end{document}

%% file: src/abstract.tex
The ParaOpt algorithm was recently introduced as a time-parallel solver for optimal-control problems with a terminal-cost objective, and convergence results have been presented for the linear diffusive case with implicit-Euler time integrators. We reformulate ParaOpt for tracking problems and provide generali\sz{}ed convergence analyses for both objectives. We focus on linear diffusive equations and prove convergence bounds that are generic in the time integrators used. For large problem dimensions, ParaOpt's performance depends crucially on having a good preconditioner to solve the arising linear systems. For the case where ParaOpt's cheap, coarse-grained propagator is linear, we introduce diagonali\sz{}ation-based preconditioners inspired by recent advances in the ParaDiag family of methods. These preconditioners not only lead to a weakly-scalable ParaOpt version, but are themselves invertible in parallel, making maximal use of available concurrency. They have proven convergence properties in the linear diffusive case that are generic in the time discreti\sz{}ation used, similarly to our ParaOpt results. Numerical results confirm that the iteration count of the iterative solvers used for ParaOpt's linear systems becomes constant in the limit of an increasing processor count. The paper is accompanied by a sequential \textsc{Matlab} implementation.

%% file: src/keywords.tex
Optimal control, ParaOpt algorithm, preconditioning, parallel-in-time

%% file: src/msc-codes.tex
49M05, 
65F08, 
65K10, 
65Y05  

%% file: src/intro/intro.tex
We study algorithms to solve the optimal-control problem
\begin{equation} \label{eq:intro:intro:optprob}
    \min_{\mv y, \mv u}J(\mv y, \mv u) \quad \text{such that} \quad \mv y'(t) = \mv g(\mv y(t)) + \mv u(t) \quad \text{and} \quad \mv y(0) = \mv{y_\mathrm{init}}    \mper
\end{equation}
This problem features a state variable $\mv y\in\mathbb R^\Sidx$ evolving over the time interval $t\in[0,T]$ with initial value $\mv{y_\mathrm{init}}$, an additive control term $\mv u\in\mathbb R^\Sidx$ and an \textsc{ode} that determines how $\mv y$ evolves under $\mv u$'s influence, including a potentially non-linear function $\mv g$. Bold-faced variables denote vectors. We call $J$ the \emph{objective function}, as we want to choose $\mv u$ to minimi\sz{}e $J$. We are interested in \emph{tracking} and \emph{terminal-cost} objectives
\begin{equation} \label{eq:intro:intro:obj}
    J(\mv y, \mv u) = \begin{cases}
        \text{Tracking:} & \frac12\int_0^T\norm{\mv y(t)-\mv{y_\mathrm d}(t)}_2^2\dif t + \frac\gamma2\int_0^T\norm{\mv u(t)}_2^2\dif t    \mcom\\
        \text{Terminal cost:} & \frac12\norm{\mv y(T)-\mv{y_\mathrm{target}}}_2^2 + \frac\gamma2\int_0^T\norm{\mv u(t)}_2^2\dif t    \mper
    \end{cases}
\end{equation}
Both use a regulari\sz{}ation parameter $\gamma>0$ that penali\sz{}es large control actions.

As can be found in, for example, \cite{ganderPARAOPTPararealAlgorithm2020a,wuDiagonalizationbasedParallelintimeAlgorithms2020b,hinzeOptimizationPDEConstraints2009b}, a minimum of the optimal-control problem \cref{eq:intro:intro:optprob} must satisfy the \emph{optimality system}
\begin{subequations} \label{eq:intro:intro:optsys}
    \begin{align}
        &\mv y'(t) = \mv g(\mv y(t)) - \mv\ad(t)/\gamma, \quad \mv y(0) = \mv{y_\mathrm{init}}\\
        &\begin{cases}
            \text{Tracking:} & \mv\ad'(t) = - \mv g'(\mv y(t))^*\mv\ad(t) + \mv{y_\mathrm d}(t) - \mv y(t), \quad \mv\ad(T) = \mv0    \mcom\\
            \text{Terminal cost:} & \mv\ad'(t) = - \mv g'(\mv y(t))^*\mv\ad(t), \quad \mv\ad(T)=\mv y(T) - \mv{y_\mathrm{target}}    \mcom
        \end{cases}
    \end{align}
\end{subequations}
which corresponds to solving a boundary value problem (\textsc{bvp}) in the state variable $\mv y(t)$ and the \emph{adjoint} variable $\mv\ad(t)\coloneqq-\gamma\mv u(t)$. We currently place no restrictions on the function $\mv g$, but from \cref{sec:conv} onward\ukus{s}{}, our analyses and preconditioners will require a linear $\mv g(\mv y(t)) = -K\mv y(t)$ where $K\in\mathbb R^{\Sidx\times\Sidx}$.

To make optimal use of increasingly parallel hardware, a lot of attention has recently been drawn to \emph{time-parallel} methods for solving initial-value problems (\textsc{ivp}s). Many time-parallel approaches to optimal control incorporate \textsc{ivp} solvers as subroutines in an optimi\sz{}ation loop \cite{gotschelEfficientParallelinTimeMethod2019a,skeneParallelintimeApproachAccelerating2021a}. This is a very general approach; however, the speed-up can be limited \cite{ganderPARAOPTPararealAlgorithm2020a}. Optimal-control problems such as \cref{eq:intro:intro:optprob}--\cref{eq:intro:intro:obj} can be solved directly with the \textsc{bvp} \cref{eq:intro:intro:optsys}. The time-parallel ParaDiag algorithm for \textsc{ivp}s \cite{mcdonaldPreconditioningIterativeSolution2018a,ganderConvergenceAnalysisPeriodiclike2019a,ganderParaDiagParallelintimeAlgorithms2021a}, for example, has been extended to these optimality systems for wave equations \cite{wuParallelInTimeBlockCirculantPreconditioner2020a} and parabolic equations \cite{wuDiagonalizationbasedParallelintimeAlgorithms2020b}. The \textsc{bvp} approach is well-suited to \emph{multiple shooting} (subdividing the time interval), which can enable faster convergence \cite{ganderPARAOPTPararealAlgorithm2020a,Riahi_Salomon_Glaser_Sugny_2016}. 

The Parareal \textsc{ivp} solver \cite{lionsResolutionEDPPar2001a} has a \textsc{bvp} equivalent in ParaOpt \cite{ganderPARAOPTPararealAlgorithm2020a}, which is the main subject of this paper. Reviewed in \cref{sec:paraopt}, the ParaOpt method decomposes the time interval into sub-intervals -- on which parallel solvers are deployed -- and ensures that the state and adjoint state on interval boundaries match through an inexact-Newton procedure. ParaOpt was formulated for terminal-cost objectives, with convergence results for linear diffusive problems, in \cite{ganderPARAOPTPararealAlgorithm2020a}. This paper extends ParaOpt to the tracking optimality system. For both objectives in \cref{eq:intro:intro:obj}, we prove bounds on ParaOpt's convergence -- for linear diffusive equations -- that are generic in the time integrators used. This stands in contrast to the existing terminal-cost bound, which specifically holds for implicit Euler. \Cref{sec:conv} contains these convergence results, although some proofs are postponed to \cref{sec:proof-tr,sec:proof-tc} for ease of presentation. We then introduce diagonali\sz{}ation-based preconditioners for ParaOpt in \cref{sec:diag}. Our bounds and the performance of our preconditioners are confirmed by numerical tests in \cref{sec:num}, after which \cref{sec:concl} concludes the paper.

%% file: src/paraopt/paraopt.tex
The ParaOpt algorithm \cite{ganderPARAOPTPararealAlgorithm2020a} for time-parallel optimal control was formulated for terminal-cost objectives. We review this method, but use a slightly more general formulation that also allows for the tracking objective.
\begin{remark}[Rescaling and circumflexes]
    For tracking, define $\mv{\widehat\ad}\coloneqq\mv\ad/\sqrt\gamma$; this rescaling introduces symmetry between the forward and backward equations of the \textsc{bvp} \cref{eq:intro:intro:optsys} and the subproblem of finding $\mv{\widehat\ad}$ is generally easier (see also \cite{wuParallelInTimeBlockCirculantPreconditioner2020a,wuDiagonalizationbasedParallelintimeAlgorithms2020b}).
    
    There will be many similarities between the tracking and terminal-cost cases throughout the paper. To highlight this and simplify notation, we use variables with a $\widehat\cdot$ diacritic whose definition depends on the objective, summari\sz{}ed in \cref{tab:paraopt:paraopt:hats}. For example, terminal cost uses $\mv{\widehat\ad}\coloneqq\mv\ad$. The other variables are defined in due course.
\end{remark}

\begin{table}
    \centering
    \begin{tabular}{c|cccc}
        & $\widehat\ad$ & $\widehat\Tidx$ & $\widehat\sigma$ & $\widehat\gamma$\\
        \hline
        Tracking & $\ad/\sqrt\gamma$ & $\Tidx - 1$ & $\DT\sigma$ & $\DT/\sqrt\gamma$\\
        Terminal cost & $\ad$ & $\Tidx$ & $\DT\sigma$ & $\DT/\gamma$
        \vspace{-.3cm}
    \end{tabular}
    \caption{
        Definitions of $\widehat\ad$, $\widehat\Tidx$, $\widehat\sigma$ and $\widehat\gamma$ for each objective function
        \vspace{-.7cm}
    }
    \label{tab:paraopt:paraopt:hats}
\end{table}

To solve the optimality system \cref{eq:intro:intro:optsys}, ParaOpt starts by dividing $[0, T]$ into $\Tidx$ sub-intervals $\{[T_{\tidx-1}, T_\tidx]\}_{\tidx=1}^\Tidx$. We will assume these to be equally large: $T_\tidx-T_{\tidx-1}=\Delta T$ for all $\tidx$. We then introduce approximations $\{\mv y_\tidx\}_{\tidx=1}^{\widehat\Tidx}$ and $\{\mv{\widehat\ad}_\tidx\}_{\tidx=1}^{\widehat\Tidx}$ to $\mv y(t=\tidx\DT)$ and $\mv{\widehat\ad}(t=\tidx\DT)$. Here, the range $1$ to $\widehat\Tidx$ (see \cref{tab:paraopt:paraopt:hats}) is justified by boundary values either being given as boundary conditions, or not being needed to compute the other unknowns. We define \emph{propagators} $\mathcal P$ and $\mathcal Q$ that approximately solve the \textsc{bvp} \cref{eq:intro:intro:optsys} on the smaller interval $[T_{\tidx-1}, T_\tidx]$ and with the altered boundary conditions $\mv y(T_{\tidx-1}) = \mv y_{\tidx-1}$ and $\mv{\widehat\ad}(T_\tidx) = \mv{\widehat\ad}_\tidx$. Specifically, we have
\begin{equation}
    \mathcal P(T_{\tidx-1}, T_\tidx, \mv y_{\tidx-1}, \mv{\widehat\ad}_\tidx) \approx \mv y_\tidx \quad \text{and} \quad \mathcal Q(T_{\tidx-1}, T_\tidx, \mv y_{\tidx-1}, \mv{\widehat\ad}_\tidx) \approx \mv{\widehat\ad}_{\tidx-1} \quad \text{for all $\tidx$}    \mcom
\end{equation}
such that $\mathcal P$ and $\mathcal Q$ propagate the state forward and adjoint state backward, respectively, over the sub-interval $[T_{\tidx-1}, T_\tidx]$. For notational conciseness, we will omit the first two arguments from $\mathcal P$ and $\mathcal Q$ when they are clear from context. The system \cref{eq:intro:intro:optsys} can be solved by composing $\mathcal P$ and $\mathcal Q$ over the sub-intervals if the boundary values are correct, and if the state and adjoint variables are continuous at interval boundaries. These are called the \emph{matching conditions}. Define the concatenations $\mv y=[\mv y_1^\trsp, \ldots, \mv y_{\widehat\Tidx}^\trsp]^\trsp$ and $\mv{\widehat\ad}=[\mv{\widehat\ad}_1^\trsp, \ldots, \mv{\widehat\ad}_{\widehat\Tidx}^\trsp]^\trsp$, and the variable
\begin{equation} \label{eq:paraopt:paraopt:Qhat}
    \widehat Q = \begin{cases}
        \text{Tracking:} & \mv y_{\widehat\Tidx}-\mv{y_\mathrm{target}}    \mcom\\
        \text{Terminal cost:} & \mathcal Q(\mv y_{\widehat\Tidx}, \mv0)    \mper
    \end{cases}
\end{equation}
Then the matching conditions can be written as the non-linear system
\begin{equation}
    \mv f\left(\left[\begin{smallmatrix}
        \mv y\\ \mv\ad\\
    \end{smallmatrix}\right]\right) \coloneqq \left[\begin{smallmatrix}
        \mv y_1 - \mathcal P(\mv{y_\mathrm{init}}, \mv\ad_1)\\
        \vdots\\
        \mv y_{\widehat\Tidx - 1} - \mathcal P(\mv y_{\widehat\Tidx - 2}, \mv \ad_{\widehat\Tidx - 1})\\
        \mv y_{\widehat\Tidx} - \mathcal P(\mv y_{\widehat\Tidx - 1}, \mv \ad_{\widehat\Tidx})\\
        \cmidrule(lr){1-1}
        \mv{\widehat\ad}_1 - \mathcal Q(\mv y_1, \mv{\widehat\ad}_2)\\
        \vdots\\
        \mv{\widehat\ad}_{\widehat\Tidx - 1} - \mathcal Q(\mv y_{\widehat\Tidx - 1}, \mv{\widehat\ad}_{\widehat\Tidx})\\
        \mv{\widehat\ad}_{\widehat\Tidx} - \widehat Q\\
    \end{smallmatrix}\right] = \kronv0    \mcom
\end{equation}
which can be solved with Newton's method (introducing iteration index $k$)
\begin{equation} \label{eq:paraopt:tc:newton}
    \mv f'\left(\left[\begin{smallmatrix}
        \mv y^{\iidx-1} \\ \mv{\widehat\ad}^{\iidx-1}\\
    \end{smallmatrix}\right]\right)\left[\begin{smallmatrix}
        \mv y^\iidx  - \mv y^{\iidx - 1}\\ \mv{\widehat\ad}^\iidx - \mv{\widehat\ad}^{\iidx - 1}\\
    \end{smallmatrix}\right] = -\mv f\left(\left[\begin{smallmatrix}
        \mv y^{\iidx-1}\\ \mv{\widehat\ad}^{\iidx-1}\\
    \end{smallmatrix}\right]\right)
\end{equation}
where the Jacobian $\mv f'\Bigl(\Bigl[\begin{smallmatrix}
    \mv y\\\mv{\widehat\ad}
\end{smallmatrix}\Bigr]\Bigr) =$
\begin{equation} \label{eq:paraopt:tc:jac}
    \left[\begin{smallarray}{cccc|cccc}
        I & & & & -\mathcal P_\ad(\mv{y_\mathrm{init}}, \mv{\widehat\ad}_1)\\
        -\mathcal P_y(\mv y_1, \mv{\widehat\ad}_2) & I & & & & -\mathcal P_\ad(\mv y_1, \mv{\widehat\ad}_2)\\
        &\ddots&\ddots&&&&\ddots\\
        && -\mathcal P_y(\mv y_{\Tidx-1},\mv{\widehat\ad}_\Tidx) & I &&&&-\mathcal P_\ad(\mv y_{\Tidx-1}, \mv{\widehat\ad}_\Tidx)\\
        \cmidrule(lr){1-4}\cmidrule(lr){5-8}
        -\mathcal Q_y(\mv y_1, \mv{\widehat\ad}_2) &&&& I & -Q_\ad(\mv y_1, \mv{\widehat\ad}_2)\\
        &\ddots&&&&\ddots&\ddots\\
        &&-\mathcal Q_y(\mv y_{\Tidx-1}, \mv{\widehat\ad}_\Tidx)&&&&I&-\mathcal Q_\ad(\mv y_{\Tidx-1},\mv{\widehat\ad}_\Tidx)\\
        &&&-\frac{\partial\widehat Q}{\partial\mv y_{\widehat\Tidx}}&&&&I\\
    \end{smallarray}\right]
\end{equation}
with $\cdot_y\coloneqq\frac{\partial\cdot}{\partial\mv{y}}$ and $\cdot_\ad\coloneqq\frac{\partial\cdot}{\partial\mv{{\widehat\ad}}}$. These derivative terms are expensive to compute; this sparks interest in an \emph{inexact Newton iteration}, which only requires an approximate Jacobian. To this end, the authors in \cite{ganderPARAOPTPararealAlgorithm2020a} replace $\mathcal P_y$, $\mathcal P_\ad$, $\mathcal Q_y$, and $\mathcal Q_\ad$ by derivatives of coarse, approximate propagators $\tilde{\mathcal P}$ and $\tilde{\mathcal Q}$. This results in an approximate Jacobian
\begin{equation} \label{eq:conv:setting:ftilde}
    \mv{\tilde f}' \approx \mv f'    \mper
\end{equation}
The ParaOpt algorithm, then, is to approximate \cref{eq:paraopt:tc:newton} -- in each iteration of the Newton procedure -- by using an inner iterative solver such as \textsc{gmres} that uses the coarse propagators $\tilde{\mathcal P}$ and $\tilde{\mathcal Q}$ to evaluate multiplications by an approximate Jacobian $\mv{\tilde f}'$. Such multiplication is embarrassingly paralleli\sz{}able, since each element in the resulting vector can be calculated independently.

%% file: src/conv/intro.tex
This section treats theoretical convergence results about ParaOpt, building on the results in \cite{ganderPARAOPTPararealAlgorithm2020a}. \Cref{sec:conv:setting} introduces the linear diffusive setting and the simplifications it allows us to make. \Cref{sec:conv:conv} then proves new, generali\sz{}ed bounds for both tracking and terminal-cost objectives.

%% file: src/conv/setting.tex
The rest of the paper considers the case of a linear $\mv g(\mv y(t)) \eqqcolon -K\mv y(t)$. The current section also assumes a symmetric matrix $K=K^\trsp\in\mathbb R^{\Sidx\times\Sidx}$ whose eigenvalues $\sigma$ are all positive (making the equation dissipative). An equation meeting these three requirements will be called \emph{linear diffusive}. Related to each property, an assumption is imposed on the propagators $\mathcal P$, $\mathcal Q$, $\tilde{\mathcal P}$, and $\tilde{\mathcal Q}$.

\begin{assumption} \label{ass:conv:setting:1}
    The propagators are affine; that is, that there exist matrices $\Phi_\mathcal P$, $\Phi_\mathcal Q$, $\Psi_\mathcal P$, $\Psi_\mathcal Q$, $\tilde\Phi_\mathcal P$, $\tilde\Phi_\mathcal Q$, $\tilde\Psi_\mathcal P$, and $\tilde\Psi_\mathcal Q$ independent of $\tidx$ such that,\begin{subequations} \label{eq:conv:setting:lin}
        \begin{align}
            {\mathcal P}(\mv y_{\tidx-1}, \mv{\widehat\ad}_\tidx) &= \Phi_\mathcal P \mv y_{\tidx-1} - \Psi_\mathcal P\mv{\widehat\ad}_\tidx + \mv{ b}_{\mathcal P,\tidx}    \quad \text{for some $\mv b_{\mathcal P,\tidx}$}\mcom\\
            {\mathcal Q}(\mv y_{\tidx-1}, \mv{\widehat\ad}_\tidx) &= \Psi_\mathcal Q \mv y_{\tidx-1} + \Phi_\mathcal Q\mv{\widehat\ad}_\tidx + \mv{ b}_{\mathcal Q,\tidx}    \quad \text{for some $\mv b_{\mathcal Q,\tidx}$}\mcom\\
            \label{eq:conv:setting:lin:Ptilde} \tilde{\mathcal P}(\mv y_{\tidx-1}, \mv{\widehat\ad}_\tidx) &= \tilde\Phi_\mathcal P \mv y_{\tidx-1} - \tilde\Psi_\mathcal P\mv{\widehat\ad}_\tidx + \mv{\tilde b}_{\mathcal P,\tidx}\quad \text{for some $\mv{\tilde b}_{\mathcal P,\tidx}$}    \mcom\\
            \label{eq:conv:setting:lin:Qtilde} \tilde{\mathcal Q}(\mv y_{\tidx-1}, \mv{\widehat\ad}_\tidx) &= \tilde\Psi_\mathcal Q \mv y_{\tidx-1} + \tilde\Phi_\mathcal Q\mv{\widehat\ad}_\tidx + \mv{\tilde b}_{\mathcal Q,\tidx}\quad \text{for some $\mv{\tilde b}_{\mathcal Q,\tidx}$}    \mper
        \end{align}
    \end{subequations}
    This assumption is satisfied for many propagators when $\mv g$ in \cref{eq:intro:intro:optprob} is linear, satisfying
    \begin{equation}
        \mv g(\mv y(t)) \eqqcolon -K\mv y(t)    \mper
    \end{equation}
\end{assumption}
\begin{proof}[Implications]
If \cref{eq:conv:setting:lin} holds, $\mv f(\bigl[\begin{smallmatrix}
    \mv y\\\mv{\widehat\ad}
\end{smallmatrix}\bigr]) = A\bigl[\begin{smallmatrix}
    \mv y\\\mv{\widehat\ad}
\end{smallmatrix}\bigr] - \mv b = \mv0 \Leftrightarrow A\bigl[\begin{smallmatrix}
    \mv y\\\mv{\widehat\ad}
\end{smallmatrix}\bigr] = \mv b$ for some $\mv b$ and
\begin{equation}
    A = \left[\begin{smallarray}{cccc|cccc}
        I&&&&\Psi_\mathcal P\\
        -\Phi_\mathcal P & I & & & & \Psi_\mathcal P\\
        & \ddots & \ddots&&&&\ddots\\
        && -\Phi_\mathcal P & I &&&&\Psi_\mathcal P\\
        \cmidrule(lr){1-4}\cmidrule(lr){5-8}
         -\Psi_\mathcal Q &&&& I & -\Phi_\mathcal Q\\
        &\ddots&&&&\ddots&\ddots\\
        &&-\Psi_\mathcal Q&&&&I&-\Phi_\mathcal Q\\
        &&&-\frac{\partial\widehat Q}{\partial\mv y_{\widehat\Tidx}}&&&&I\\
    \end{smallarray}\right]
\end{equation}
with $\widehat Q$ from \cref{eq:paraopt:paraopt:Qhat}. Recalling $\mv{\tilde f}'$ from \cref{eq:conv:setting:ftilde}, we have $\mv{\tilde f}'(\bigl[\begin{smallmatrix}
    \mv y\\\mv{\widehat\ad}
\end{smallmatrix}\bigr]) = \tilde A$, with $\tilde A$ defined analogously to $A$ (using the $\tilde\Phi$s and $\tilde\Psi$s). Then the ParaOpt Newton iteration becomes
\begin{equation} \label{eq:conv:setting:paraopt-it}
    \tilde A\begin{litmat}
        \mv y^\iidx - \mv y^{\iidx-1}\\\mv{\widehat\ad}^\iidx - \mv{\widehat\ad}^{\iidx-1}
    \end{litmat} = -A\begin{litmat}
        \mv y^{\iidx-1}\\\mv{\widehat\ad}^{\iidx-1}
    \end{litmat} + \mv b
\end{equation}
and thus, if we define $\mv x^\iidx\coloneqq[(\mv y^\iidx)^\trsp, (\mv{\widehat\ad}^\iidx)^\trsp]^\trsp$,
\begin{equation} \label{eq:conv:setting:paraopt-it-x}
    \mv x^\iidx = (I - \tilde A^{-1}A)\mv x^{\iidx-1}+\tilde A^{-1}\mv b    \mper
\end{equation}
To summari\sz{}e, instead of computing $\mv x = A^{-1}\mv b$ directly, ParaOpt applied to linear equations hopes to obtain a speed-up by using the iteration \cref{eq:conv:setting:paraopt-it-x}: a series of paralleli\sz{}able multiplications by $A$, combined with inverses of $\tilde A$ (which are cheaper, as they correspond to the coarse propagators). We call $S\coloneqq(I-\tilde A^{-1}A)$ the \emph{ParaOpt iteration matrix}, as it characteri\sz{}es the evolution of the error. Indeed, denote by $\mv e^\iidx\coloneqq \mv x^\iidx-A^{-1}\mv b$ the error on the solution at iteration $k$. Then \cref{eq:conv:setting:paraopt-it-x} means that
\begin{equation}
\begin{aligned}
    \mv e^\iidx &= (I-\tilde A^{-1}A)\mv x^{\iidx-1}+\tilde A^{-1}\mv b - A^{-1}\mv b\\
                &= \mv e^{\iidx-1}-(\tilde A^{-1}A\mv x^{\iidx-1}-\tilde A^{-1}\mv b) = S\mv e^{\iidx-1}    \mper
\end{aligned}
\end{equation}
Thus, ParaOpt converges if $S$'s eigenvalues are all smaller than $1$ in magnitude.
\end{proof}

\begin{assumption} \label{ass:conv:setting:2}
    We use $O$ to denote a zero matrix. It holds that
    \begin{equation}
    \begin{cases}
        \text{Tracking:} \quad & \Phi_\mathcal P = \Phi_\mathcal Q \eqqcolon \Phi \quad \text{and} \quad \Psi_\mathcal P = \Psi_\mathcal Q \eqqcolon \Psi    \mcom\\
        \text{Terminal cost:} \quad & \Phi_\mathcal P = \Phi_\mathcal Q \eqqcolon \Phi \quad \text{and} \quad \Psi_\mathcal Q = O, \Psi_\mathcal P \eqqcolon \Psi    \mcom
    \end{cases}
    \end{equation}
    and analogously for the coarse $\tilde\Phi$ and $\tilde\Psi$. In addition, $K$, $\Phi$, $\Psi$, $\tilde\Phi$, and $\tilde\Psi$ are simultaneously diagonali\sz{}able. We denote their corresponding eigenvalues by $\sigma$, $\varphi$, $\psi$, $\tilde\varphi$, and $\tilde\psi$. These assumptions are satisfied for many propagators when $K=K^\trsp$ is symmetric.
\end{assumption}
\begin{proof}[Implications]
    ParaOpt's convergence is characteri\sz{}ed by the \emph{convergence factor} $\rho$, which is the spectral radius of $S$. Using simultaneous diagonali\sz{}ability, a similarity transform allows decomposing $S$ into $\Sidx$ smaller matrices $S_\sigma$ -- one for each of $K$'s eigenvalues $\sigma$ -- whose combined eigenvalues give those of $S$ itself (see\ukus{}{,} e.g.\ukus{}{,}\ \cite{ganderPARAOPTPararealAlgorithm2020a}). Then
    \begin{equation}
        \label{eq:conv:setting:rhobound} \rho = \max_{\sigma\in\mathrm{eig}(K)} \max(\kern2pt\abs{\mathrm{eig}(S_\sigma)})    \mcom
    \end{equation}
    where the absolute value operates element-wise. For tracking, each $S_\sigma$ has the form
    \begin{subequations} \label{eq:conv:setting:Ssigma}
    \begin{align}
        \label{eq:conv:setting:Ssigma:tr} S_\sigma &= I - \left[\begin{smallarray}{cccc|cccc}
            1 &&&& \tilde\psi\\
            -\tilde\varphi & 1 &&&& \tilde\psi\\
            & \ddots & \ddots &&&& \ddots\\
            && -\tilde\varphi & 1 &&&& \tilde\psi\\
            \cmidrule(lr){1-4}\cmidrule(lr){5-8}
            -\tilde\psi &&&& 1 & -\tilde\varphi\\
            & \ddots &&&& \ddots & \ddots\\
            && -\tilde\psi &&&& 1 & -\tilde\varphi\\
            &&& -\tilde\psi &&&& 1
        \end{smallarray}\right]^{-1}\left[\begin{smallarray}{cccc|cccc}
            1 &&&& \mathrlap{\phantom{\tilde\psi}}\psi\\
            -\varphi & 1 &&&& \mathrlap{\phantom{\tilde\psi}}\psi\\
            & \ddots & \ddots &&&& \ddots\\
            && -\varphi & 1 &&&& \mathrlap{\phantom{\tilde\psi}}\psi\\
            \cmidrule(lr){1-4}\cmidrule(lr){5-8}
            -\mathrlap{\phantom{\tilde\psi}}\psi &&&& 1 & -\varphi\\
            & \ddots &&&& \ddots & \ddots\\
            && -\mathrlap{\phantom{\tilde\psi}}\psi &&&& 1 & -\varphi\\
            &&& -\mathrlap{\phantom{\tilde\psi}}\psi &&&& 1
        \end{smallarray}\right]
    \end{align}
    and for terminal cost
    \begin{align}
        \label{eq:conv:setting:Ssigma:tc} S_\sigma &= I - \left[\begin{smallarray}{cccc|cccc}
            1 &&&& \tilde\psi\\
            -\tilde\varphi & 1 &&&& \tilde\psi\\
            & \ddots & \ddots &&&& \ddots\\
            && -\tilde\varphi & 1 &&&& \tilde\psi\\
            \cmidrule(lr){1-4}\cmidrule(lr){5-8}
            &&&& 1 & -\tilde\varphi\\
            & &&&& \ddots & \ddots\\
            &&  &&&& 1 & -\tilde\varphi\\
            &&& -1 &&&& 1
        \end{smallarray}\right]^{-1}\left[\begin{smallarray}{cccc|cccc}
            1 &&&& \mathrlap{\phantom{\tilde\psi}}\psi\\
            -\varphi & 1 &&&& \mathrlap{\phantom{\tilde\psi}}\psi\\
            & \ddots & \ddots &&&& \ddots\\
            && -\varphi & 1 &&&& \mathrlap{\phantom{\tilde\psi}}\psi\\
            \cmidrule(lr){1-4}\cmidrule(lr){5-8}
            &&&& 1 & -\varphi\\
            & &&&& \ddots & \ddots\\
            &&  &&&& 1 & -\varphi\\
            &&& -1 &&&& 1
        \end{smallarray}\right]    \mper
    \end{align}
    \end{subequations}
    Here, $\varphi$, $\psi$, $\tilde\varphi$, and $\tilde\psi$ depend on the problem properties, the time discreti\sz{}ation and $K$'s eigenvalue $\sigma$. We detail them for some relevant time discreti\sz{}ations shortly.
\end{proof}

\begin{assumption} \label{ass:conv:setting:3}
    All $\varphi$, $\psi$, $\tilde\varphi$, and $\tilde\psi$ in \cref{ass:conv:setting:2} satisfy $0<\varphi,\tilde\varphi<1$ and $0<\psi,\tilde\psi$. This is satisfied for many propagators when $\sigma>0$ for all $\sigma\in\mathrm{eig}(K)$.
\end{assumption}

For any time step $\tau$, define rescalings of $\sigma$ from \cref{ass:conv:setting:2} and $\gamma$ from \cref{eq:intro:intro:obj}:
\begin{equation}
    \begin{cases}
        \text{Tracking:} & \widehat\sigma_\tau \coloneqq \tau\sigma \quad \text{and} \quad \widehat\gamma_\tau \coloneqq \tau/\sqrt\gamma    \mcom\\
        \text{Terminal cost:} & \widehat\sigma_\tau \coloneqq \tau\sigma \quad \text{and} \quad \widehat\gamma_\tau \coloneqq \tau/\gamma    \mper
    \end{cases}
\end{equation}
Recall that $\mathcal P$, $\mathcal Q$, $\tilde{\mathcal P}$, and $\tilde{\mathcal Q}$ propagate over a time interval $\DT$ and define $\widehat\sigma \coloneqq \widehat\sigma_\DT$ and $\widehat\gamma \coloneqq \widehat\gamma_\DT$ for conciseness, as is also shown in \cref{tab:paraopt:paraopt:hats}.

\begin{remark}[Optimi\sz{}ation and discreti\sz{}ation] \label{rem:conv:setting:optdisc}
    Optimal-control algorithms either target a discreti\sz{}ation of the optimal continuous solution (first-optimi\sz{}e-then-discreti\sz{}e -- \textsc{fotd}) or optimi\sz{}e the discreti\sz{}ed problem (\textsc{fdto}) \cite{hinzeOptimizationPDEConstraints2009b}. In general, these approaches are not equivalent. The terminal-cost implicit-Euler propagators in \cite{ganderPARAOPTPararealAlgorithm2020a} are based on \textsc{fdto}; we will consider both options here. For tracking, we only use \textsc{fotd}, as \textsc{fdto}-based tracking propagators do not satisfy \cref{ass:conv:setting:2}.
\end{remark}

We can now give some values of $\varphi$ and $\psi$ for various propagators and objectives.
\begin{lemma} \label{lmm:conv:setting:prop-tr-ie}
    For tracking objectives, (\textsc{fotd}) implicit Euler with $J$ length-$\tau$ steps (i.e.\ukus{}{,} $\DT = J\tau$) applied to a linear diffusive problem satisfies \cref{ass:conv:setting:1,ass:conv:setting:2,ass:conv:setting:3} with
    \begin{equation} \label{eq:lmm:conv:setting:prop-tr-ie:1}
        \varphi = \varphi_\tau^{(J)} \quad \text{and} \quad \psi = \psi_\tau^{(J)}
    \end{equation}
    where $\varphi_\tau^{(0)}=1$ and $\psi_\tau^{(0)}=0$ and, with $\zeta \coloneqq (1+\sigma_\tau)$, we have the recursion
    \begin{equation} \label{eq:lmm:conv:setting:prop-tr-ie:2}
        \varphi_\tau^{(j+1)} = \varphi_\tau^{(j)}(\zeta^{-1}-\widehat\gamma_\tau(\psi_\tau^{(j+1)}-\widehat\gamma_\tau\zeta^{-1})) \text{ and } \psi_\tau^{(j+1)} = \frac{\widehat\gamma_\tau+\zeta^{-1}(1+\widehat\gamma_\tau^2)\psi_\tau^{(j)}}{\zeta+\widehat\gamma_\tau\psi_\tau^{(j)}}    \mper
    \end{equation}
\end{lemma}
\begin{proof}
    This result is derived in \cref{sec:apdx-po-prop:prop-tr-ie}.
\end{proof}
\begin{lemma} \label{lmm:conv:setting:prop-tr-ex}
    For tracking objectives, exact solvers on a linear diffusive problem satisfy \cref{ass:conv:setting:1,ass:conv:setting:2,ass:conv:setting:3} with
    \begin{equation}
    \begin{gathered}
        \varphi = d^{-1} \quad \text{and} \quad \psi = -d^{-1}c    \mcom\\
        \text{where } c = -\widehat\gamma\frac{\sinh(\sqrt{\widehat\gamma^2 + \widehat\sigma^2})}{\sqrt{\widehat\gamma^2 + \widehat\sigma^2}} \text{ and } d = \cosh(\sqrt{\widehat\gamma^2 + \widehat\sigma^2}) + \widehat\sigma\frac{\sinh(\sqrt{\widehat\gamma^2 + \widehat\sigma^2})}{\sqrt{\widehat\gamma^2 + \widehat\sigma^2}}    \mper
    \end{gathered}
    \end{equation}
\end{lemma}
\begin{proof}
    This result is derived in \cref{sec:apdx-po-prop:prop-tr-ex}.
\end{proof}
\begin{lemma} \label{lmm:conv:setting:prop-tc-ie-fdto}
    For terminal cost, \textsc{fdto} implicit Euler with $J$ length-$\tau$ steps (i.e.\ukus{}{,} $\DT = J\tau$) applied to a linear diffusive problem satisfies \cref{ass:conv:setting:1,ass:conv:setting:2,ass:conv:setting:3} with
    \begin{equation}
        \varphi = (1+\sigma\tau)^{-J} \quad \text{and} \quad \psi = \frac{1-\varphi^2}{\gamma\sigma(2+\sigma\tau)}    \mper
    \end{equation}
\end{lemma}
\begin{proof}
    This result is found in \cite[(3.16)--(3.17)]{ganderPARAOPTPararealAlgorithm2020a}.
\end{proof}
\begin{lemma} \label{lmm:conv:setting:prop-tc-ie-fotd}
    For terminal cost, \textsc{fotd} implicit Euler with $J$ length-$\tau$ steps (i.e.\ukus{}{,} $\DT = J\tau$) applied to a linear diffusive problem satisfies \cref{ass:conv:setting:1,ass:conv:setting:2,ass:conv:setting:3} with
    \begin{equation} \label{eq:lmm:conv:setting:prop-tc-ie-fotd}
        \varphi = (1+\sigma\tau)^{-J} \quad \text{and} \quad \psi = \frac{(1-\varphi_\tau^2)(1+\widehat\sigma_\tau)}{\gamma\sigma(2+\sigma\tau)}    \mper
    \end{equation}
\end{lemma}
\begin{proof}
    This result can be derived analogously to \cref{lmm:conv:setting:prop-tr-ie}.
\end{proof}
\begin{lemma} \label{lmm:conv:setting:prop-tc-ex}
    For terminal cost, exact solvers on a linear diffusive problem satisfy \cref{ass:conv:setting:1,ass:conv:setting:2,ass:conv:setting:3} with
    \begin{equation} \label{eq:lmm:conv:setting:prop-tc-ex}
        \varphi = d^{-1} \quad \text{and} \quad \psi = -d^{-1}b, \quad \text{where} \quad b=-\widehat\gamma\frac{\sinh\widehat\sigma}{\widehat\sigma} \quad \text{and} \quad d = \exp\widehat\sigma    \mper
    \end{equation}
\end{lemma}
\begin{proof}
    This result can be derived analogously to \cref{lmm:conv:setting:prop-tr-ex}.
\end{proof}

%% file: src/conv/conv.tex
Terminal-cost ParaOpt was proposed in \cite{ganderPARAOPTPararealAlgorithm2020a}, which includes a convergence bound for the case where both the fine and the coarse propagators use \textsc{fdto} implicit Euler. We propose alternative bounds $\rho^*$ on ParaOpt's convergence factor $\rho$ that are generic in the propagators used. \Cref{thm:conv:conv:tr-gen,thm:conv:conv:tc-gen} treat tracking and terminal cost, respectively.

\begin{theorem} \label{thm:conv:conv:tr-gen}
   When tracking ParaOpt is applied to a linear diffusive equation and \cref{ass:conv:setting:1,ass:conv:setting:2,ass:conv:setting:3} hold, the convergence factor $\rho$ satisfies
    \begin{equation} \label{eq:thm:conv:conv:tr-gen:bound}
        \rho < \rho^* \coloneqq \max_{\sigma\in\mathrm{eig}(K)}\sqrt{\frac{(\tilde\varphi - \varphi)^2 + (\tilde\psi - \psi)^2}{(1-\tilde\varphi)^2 + \tilde\psi^2}}    \mper
    \end{equation}
\end{theorem}
\begin{proof}
    \Cref{sec:proof-tr} is dedicated to proving this result.
\end{proof}

\clearpage

The $\varphi$, $\psi$, $\tilde\varphi$, and $\tilde\psi$ values for a specific set of propagators can be filled in to the general bound given in \cref{thm:conv:conv:tr-gen}. We give an example.
\begin{example} \label{ex:conv:conv:special}
    Consider tracking ParaOpt applied to a linear diffusive equation with an exact fine propagator and a one-step implicit-Euler coarse one. Then we can combine the results from \cref{lmm:conv:setting:prop-tr-ex,lmm:conv:setting:prop-tr-ie} with \cref{eq:thm:conv:conv:tr-gen:bound} to bound $\rho$. In fact, in this specific case, we can prove that
    \begin{equation} \label{eq:conv:conv:special}
        \rho < \rho^* < 1
    \end{equation}
    for any $\sigma$, $\gamma$, and $\DT$, meaning that ParaOpt does not diverge. The bound \cref{eq:conv:conv:special} is proven in \cref{sec:apdx-po-prop:special}.
\end{example}

\begin{theorem} \label{thm:conv:conv:tc-gen}
    When terminal-cost ParaOpt is applied to a linear diffusive equation and \cref{ass:conv:setting:1,ass:conv:setting:2,ass:conv:setting:3} hold, the convergence factor $\rho$ satisfies
    \begin{equation} \label{eq:thm:conv:conv:tc-gen:bound}
        \rho \le \rho^* \coloneqq \max_{\sigma\in\mathrm{eig}(K)}\max\left(\frac{\abs{\varphi-\tilde\varphi}}{1-\tilde\varphi}, x^*\right)    \mcom
    \end{equation}
    where $x^*$ is the root of $f_\infty(x) \coloneqq \frac{\tilde\psi-\psi}{\tilde\psi+1/\sum_{\tidx=0}^\infty(\tilde\varphi+\frac{\varphi-\tilde\varphi}x)^{2\tidx}}-x$ with the largest magnitude.
\end{theorem}
\begin{proof}
    \Cref{sec:proof-tc} is dedicated to proving this result.
\end{proof}
Note that, while its definition is implicit, $x^*$ in \cref{thm:conv:conv:tc-gen} is efficiently computable. This is elaborated upon at the end of \cref{sec:proof-tc}.

%% file: src/conv/interp.tex
The $\rho$ bounds \cref{eq:thm:conv:conv:tc-gen:bound,eq:thm:conv:conv:tr-gen:bound} are generic in the propagators used (under \cref{ass:conv:setting:1,ass:conv:setting:2,ass:conv:setting:3}) and independent of $\widehat\Tidx$. The latter property not only results in efficiently computable bounds (as opposed to calculating eigenvalues of potentially large matrices), but also ensures that the number of ParaOpt iterations stays constant when increasing $T$ together with $\widehat\Tidx$. We return to this scaling in \cref{sec:diag:scale}.

\begin{figure}
    \begin{subfigure}[b]{.4\textwidth}
        \includegraphics[width=\textwidth]{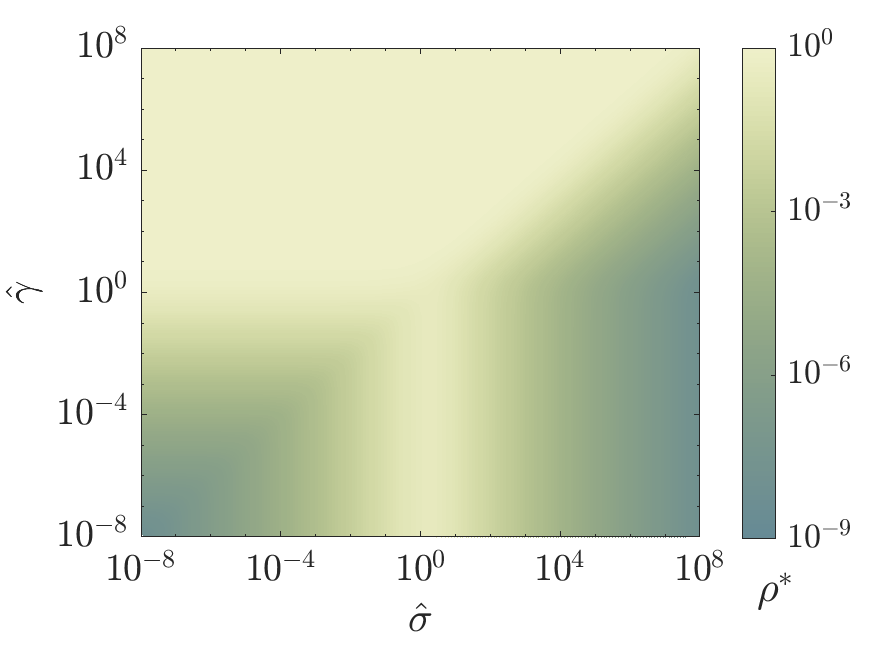}
        \caption{
            Tracking, $J=1$
            \vspace{-.4cm}
        }
        \label{fig:conv:interp:rhostar:tr-1}
    \end{subfigure}
    \hfill
    \begin{subfigure}[b]{.4\textwidth}
        \includegraphics[width=\textwidth]{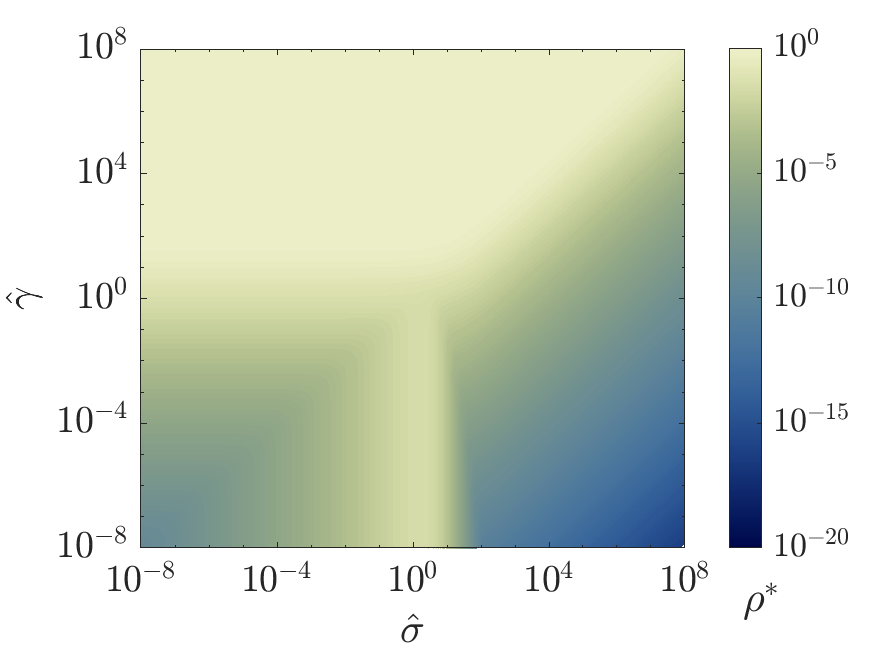}
        \caption{
            Tracking, $J=10$
            \vspace{-.4cm}
        }
        \label{fig:conv:interp:rhostar:tr-10}
    \end{subfigure}
    \vfill
    \begin{subfigure}[b]{.4\textwidth}
        \includegraphics[width=\textwidth]{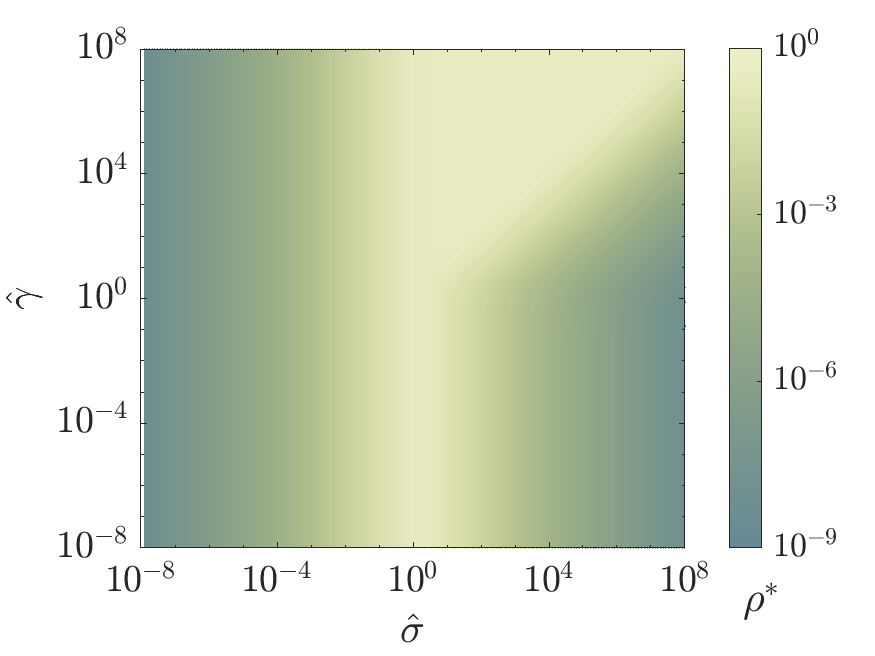}
        \caption{
            Terminal cost, \textsc{fotd}, $J=1$
            \vspace{-.4cm}
        }
        \label{fig:conv:interp:rhostar:tc-fotd-1}
    \end{subfigure}
    \hfill
    \begin{subfigure}[b]{.4\textwidth}
        \includegraphics[width=\textwidth]{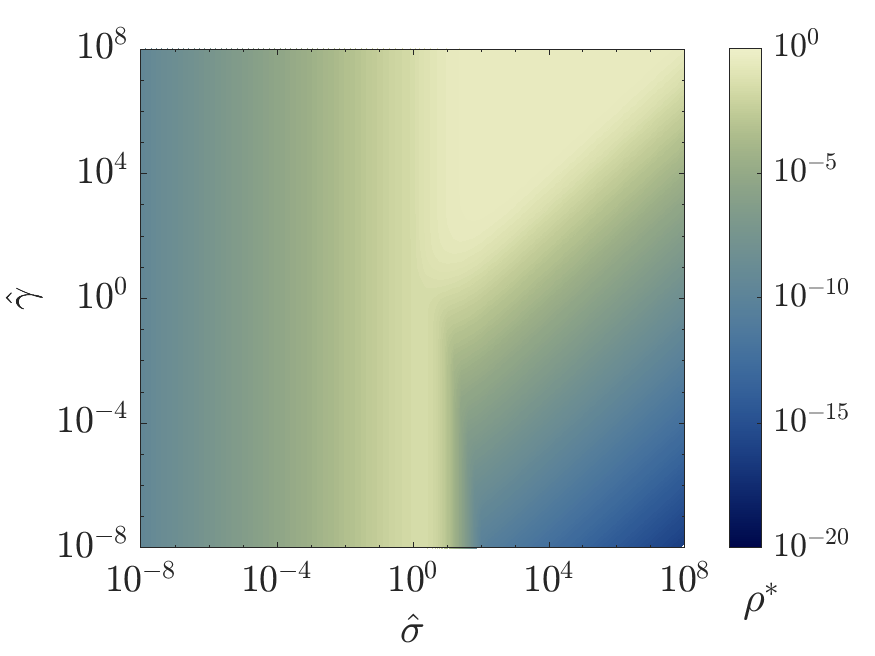}
        \caption{
            Terminal cost, \textsc{fotd}, $J=10$
            \vspace{-.4cm}
        }
        \label{fig:conv:interp:rhostar:tc-fotd-10}
    \end{subfigure}
    \vfill
    \begin{subfigure}[b]{.4\textwidth}
        \includegraphics[width=\textwidth]{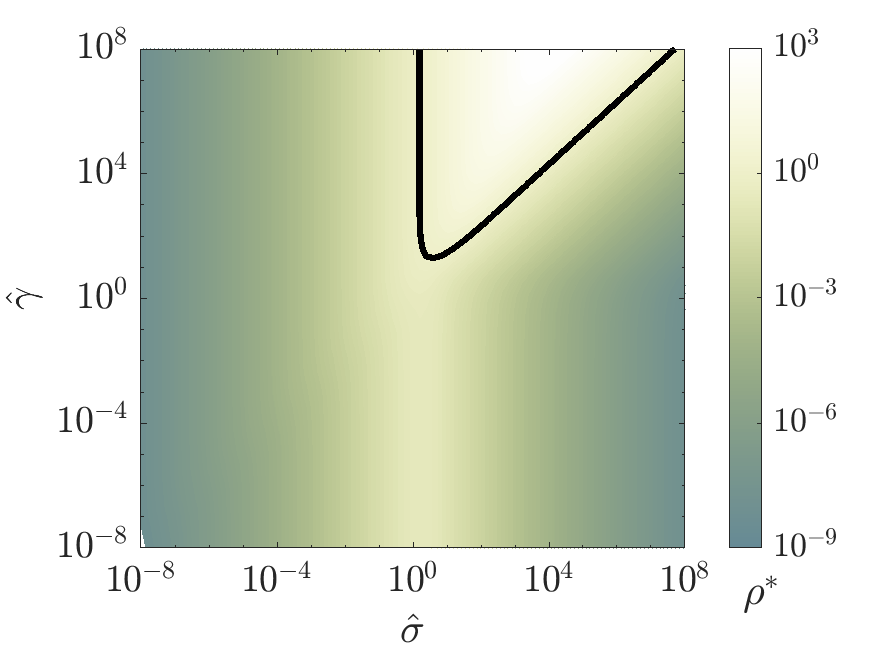}
        \caption{
            Terminal cost, \textsc{fdto}, $J=1$
            \vspace{-.5cm}
        }
        \label{fig:conv:interp:rhostar:tc-fdto-1}
    \end{subfigure}
    \hfill
    \begin{subfigure}[b]{.4\textwidth}
        \includegraphics[width=\textwidth]{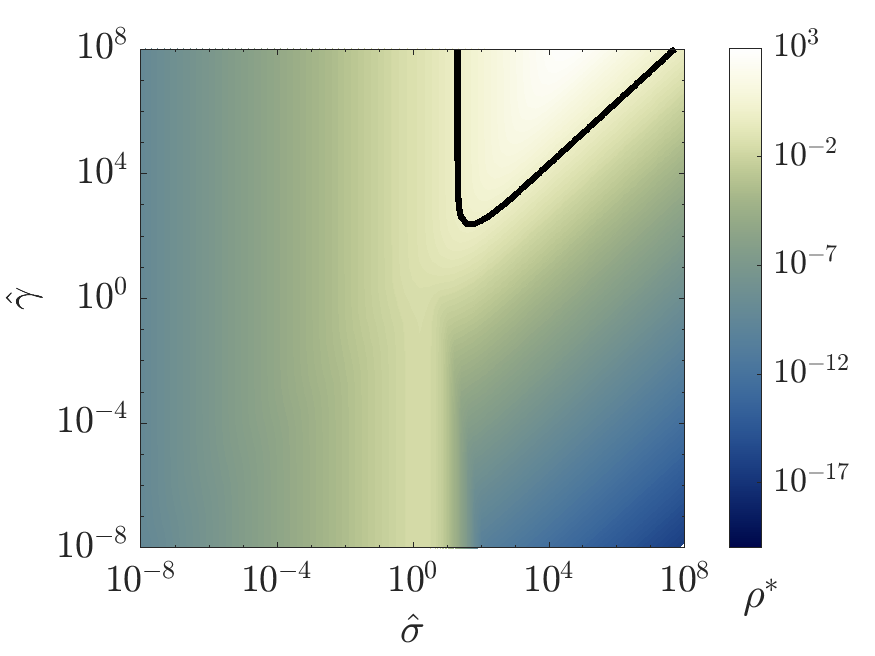}
        \caption{
            Terminal cost, \textsc{fdto}, $J=10$
            \vspace{-.5cm}
        }
        \label{fig:conv:interp:rhostar:tc-fdto-10}
    \end{subfigure}
    \caption{
        The bound \cref{eq:thm:conv:conv:tr-gen:bound} or \cref{eq:thm:conv:conv:tc-gen:bound} on ParaOpt's convergence factor $\rho^*$ is shown, with an exact fine and a $J$-step implicit-Euler coarse propagator. Recall that $\rho^* < 1$ guarantees convergence. The black contour lines mark $\rho^*=1$.
        \vspace{-.9cm}
    }
    \label{fig:conv:interp:rhostar}
\end{figure}
\begin{figure}[b]
    \begin{minipage}{.4\textwidth}
        \includegraphics[width=\textwidth]{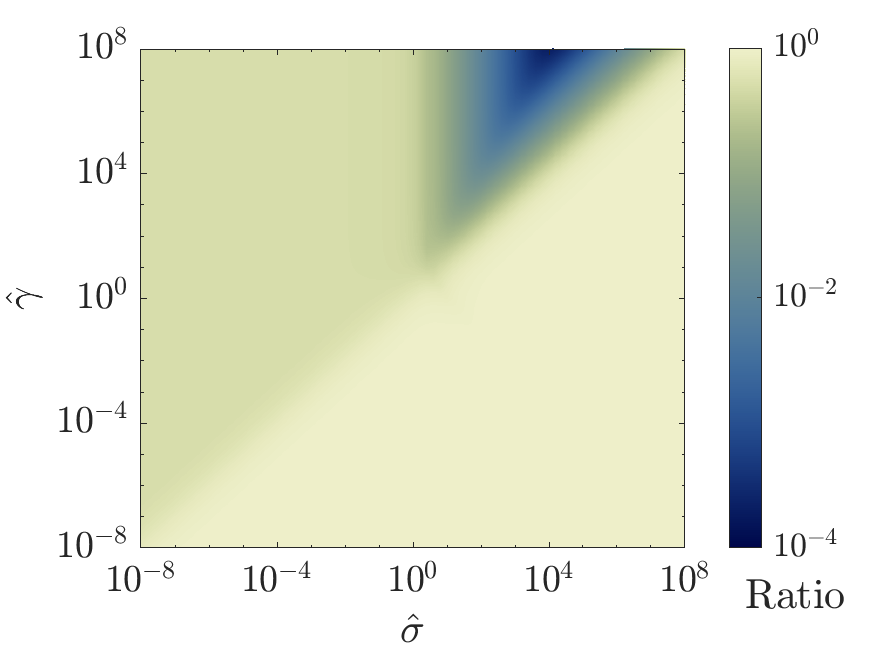}
        \caption{
            Ratio of the $\rho^*$ values in \cref{fig:conv:interp:rhostar:tc-fotd-1,fig:conv:interp:rhostar:tc-fdto-1}
        }
        \label{fig:conv:interp:rhostar:fdtovsfotd}
    \end{minipage}
    \hfill
    \begin{minipage}{.4\textwidth}
        \includegraphics[width=\textwidth]{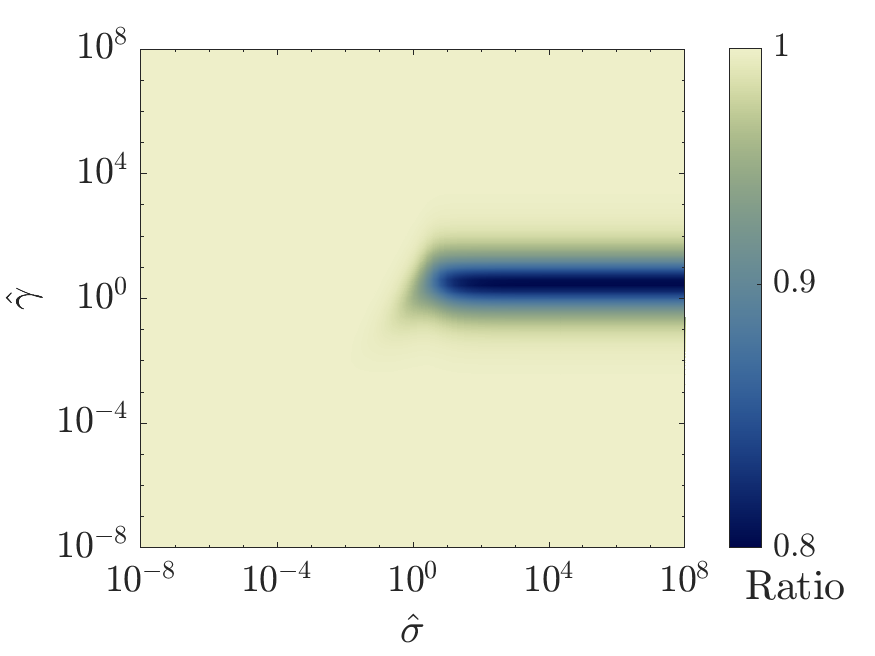}
        \caption{
            Ratio of $\rho^*$ in \cref{fig:conv:interp:rhostar:tc-fdto-1} with that of the bound in \cite{ganderPARAOPTPararealAlgorithm2020a}
        }
        \label{fig:conv:interp:rhostar:origvsnew}
    \end{minipage}
\end{figure}

To investigate specific propagators, their $\varphi$ and $\psi$ can be calculated as a function of the problem parameters $\sigma$, $\gamma$, and $\tau$. \Cref{sec:apdx-po-prop:phipsi} goes into more detail about this procedure. A bound on $\rho$ can then be obtained by filling $\varphi$ and $\psi$ into \cref{eq:thm:conv:conv:tc-gen:bound,eq:thm:conv:conv:tr-gen:bound}. Let us do so for an exact fine propagator (or, equivalently, a numerical propagator in the limit for infinitely small time steps) and an implicit-Euler coarse propagator. Recall from \cref{rem:conv:setting:optdisc} that, in the case of terminal cost, two implicit-Euler discreti\sz{}ations are possible: either first-discreti\sz{}e-then-optimi\sz{}e (as \cite{ganderPARAOPTPararealAlgorithm2020a} uses) or the other way around. We will compare both techniques.

\Cref{fig:conv:interp:rhostar} shows $\rho^*$ as a function of the problem parameters $\widehat\sigma$ and $\widehat\gamma$. Recall that $\rho^*$ is an upper bound on ParaOpt's convergence factor; $\rho^*<1$ means guaranteed convergence and, as $\rho^*$ decreases further, this convergence accelerates. In the tracking case, \cref{fig:conv:interp:rhostar:tr-1} confirms the result \cref{eq:conv:conv:special}, which guarantees $\rho$ never exceeds $1$. For terminal cost, \cref{fig:conv:interp:rhostar:tc-fotd-1,fig:conv:interp:rhostar:tc-fotd-10} show the bound \cref{eq:thm:conv:conv:tc-gen:bound} for \textsc{fotd} implicit Euler; \cref{fig:conv:interp:rhostar:tc-fdto-1,fig:conv:interp:rhostar:tc-fdto-10} concern \textsc{fdto}. With an exact fine propagator, it is clear that our \textsc{fotd} coarse propagators never cause divergence for linear diffusive problems, while those based on \textsc{fdto} might. \Cref{fig:conv:interp:rhostar:fdtovsfotd} shows the ratio of the bounds $\rho^*$ from \cref{fig:conv:interp:rhostar:tc-fotd-1,fig:conv:interp:rhostar:tc-fdto-1}, comparing the \textsc{fotd} and \textsc{fdto} strategies. This analysis showcases the advantage of having generic results: we can study these two propagators side-by-side, while previous bounds were specific to a single propagator choice. To assess the quality of our upper bound, we compare it to the one given by \cite[Corollary 3.6 and Theorem 3.8]{ganderPARAOPTPararealAlgorithm2020a} in \cref{fig:conv:interp:rhostar:origvsnew}. This shows the novel bound to be at least as tight as the existing one.

%% file: src/diag/intro.tex
Having presented the ParaOpt algorithm and studied its convergence, we will now consider its scaling in \cref{sec:diag:scale}, establishing the need for efficient preconditioners in the inexact-Newton step. For affine coarse propagators, we propose in \cref{sec:diag:linear,sec:diag:prec} to construct a preconditioner that uses \emph{alpha-circulant} approximations of the system matrix, which can be inverted efficiently by a diagonali\sz{}ation procedure. \Cref{sec:diag:small} then details how to solve the arising smaller linear systems, after which \cref{sec:diag:conv} mentions some properties of the preconditioners that influence the convergence of iterative solvers.

%% file: src/diag/scale.tex
When ParaOpt uses time integrators that satisfy \cref{ass:conv:setting:1,ass:conv:setting:2,ass:conv:setting:3}, $\rho^*$ in \cref{thm:conv:conv:tc-gen,thm:conv:conv:tr-gen} is an upper bound on the rate of the algorithm's exponential convergence. When the number of intervals $\Tidx$ is increased -- which scales both the problem size and the available parallelism -- we study the evolution of $\rho^*$ in two regimes of \emph{weak scaling} \cite{ganderPARAOPTPararealAlgorithm2020a}.
 \begin{itemize}
    \item When $T$ and $\Tidx$ are increased together, $\widehat\sigma$ and $\widehat\gamma$ do not change. Given that our bounds are $\Tidx$-independent, $\rho^*$ stays constant and the upper bound on the number of ParaOpt iterations is independent of the available parallelism.
    \item When $T$ is kept constant, $\DT$ -- and with it, $\widehat\sigma$ and $\widehat\gamma$ -- decreases instead. Scaling $\Tidx$ then corresponds to traversing the graphs in \cref{fig:conv:interp:rhostar} in a diagonal line towards the region with small $\widehat\sigma$ and $\widehat\gamma$. In, \cite{ganderPARAOPTPararealAlgorithm2020a}, it already noted that $\rho^*$ is bounded in this scaling regime for \textsc{fdto} implicit-Euler propagators on a linear diffusive terminal-cost problem. Our graphs allow similar conclusions and can be drawn for any propagators satisfying \cref{ass:conv:setting:1,ass:conv:setting:2,ass:conv:setting:3}. In \cref{fig:conv:interp:rhostar}'s examples, $\rho^*$ even goes to zero in the scaling limit, surpassing the usual concept of weak scalability. In the non-asymptotic regime, weak scalability may be absent until the maximum at $\widehat\sigma\approx1$ has been surmounted.
\end{itemize}
It follows from this discussion that, if the amount of work in each ParaOpt iteration scales linearly with $\Tidx$, the algorithm is weakly scalable for linear diffusive problems (and, experimentally, more broadly \cite{ganderPARAOPTPararealAlgorithm2020a}). This condition, however, is not yet fulfilled. The inexact-Newton procedure solves, in each iteration, a linear system with a matrix of size $(2\Sidx\widehat\Tidx\times2\Sidx\widehat\Tidx)$. With a direct or non-preconditioned iterative method, the cost of this scales superlinearly with $\widehat\Tidx$ (and thus with $\Tidx$), which -- especially if $K\in\mathbb R^{\Sidx\times\Sidx}$ is large but sparse -- starts to dominate the algorithm's execution time.

Hence our mission in the rest of this section is to precondition these systems such that the cost of solving them with iterative methods is linear in $\Tidx$. Each iteration of the iterative solver uses a matrix multiplication and an inversion, which will both scale (log-)linearly with $\Tidx$. The number of iterations should then be constant.

%% file: src/diag/linear.tex
Our preconditioners will apply to affine coarse propagators (that is, $\tilde{\mathcal P}$ and $\tilde{\mathcal Q}$ are of the forms \cref{eq:conv:setting:lin:Ptilde,eq:conv:setting:lin:Qtilde}). Then the inexact-Newton step -- \emph{coarse-grid correction}, in Parareal vernacular -- looks like
\begin{equation}
    \tilde A\left[\begin{smallmatrix}
        \mv y^\iidx - \mv y^{\iidx-1}\\\mv{\widehat\ad}^\iidx - \mv{\widehat\ad}^{\iidx-1}
    \end{smallmatrix}\right] = -\mv f\left(\left[\begin{smallmatrix}
        \mv y^{\iidx-1}\\\mv{\widehat\ad}^{\iidx-1}
    \end{smallmatrix}\right]\right)
\end{equation}
where the matrix $\tilde A$ can be written as (using $\otimes$ to denote a Kronecker product)
\begin{equation} \label{eq:diag:linear:Atilde}
    \tilde A = \begin{litmat}
            I \otimes I + B \otimes \tilde\Phi_\mathcal P & I \otimes \tilde\Psi_\mathcal P\\
            -I \otimes \tilde\Psi_\mathcal Q + E \otimes (\tilde\Psi_\mathcal Q - \partial\widehat Q/\partial\mv y_{\widehat\Tidx}) & I \otimes I + B^\trsp \otimes \tilde\Phi_\mathcal Q
        \end{litmat}    \mper
\end{equation}
Here, $\widehat Q$ is as defined in \cref{eq:paraopt:paraopt:Qhat}, $E$ is a matrix with as only non-zero a one in its bottom-right corner, and $B$ only has $(-1)$s on its first sub-diagonal.

%% file: src/diag/prec.tex
Inspired by work on paralleli\sz{}ing the coarse-grid correction of Parareal \cite{wuParallelCoarseGrid2018a}, we propose a low-rank perturbation of $\tilde A$ in \cref{eq:diag:linear:Atilde} as a preconditioner. Introduce a parameter $\alpha\in\mathbb C$ and define
\begin{equation} \label{eq:diag:prec:Palpha}
    P(\alpha) \coloneqq \begin{litmat}
        I \otimes I + C(\alpha) \otimes \tilde\Phi_\mathcal P & I \otimes \tilde\Psi_\mathcal P\\
        -I \otimes \tilde\Psi_\mathcal Q & I \otimes I + C^*(\alpha) \otimes \tilde\Phi_\mathcal Q
    \end{litmat}    \mcom
\end{equation}
where $C(\alpha)$ differs from $B$ only by an additional non-zero value of $-\alpha$ in the top-right corner. This preconditioner contains two modifications compared to $\tilde A$: \begin{itemize}
    \item Replacing $B$ by $C(\alpha)$ causes a rank-$(2\Sidx)$ perturbation, which is scaled by $\alpha$.
    \item The term $E\otimes(\tilde\Psi_\mathcal Q-\partial\widehat Q/\partial\mv y_{\widehat\Tidx})$ has been left out. This has no effect in the tracking case, but is a rank-$\Sidx$ perturbation for terminal cost (although it combines with the perturbation above to another rank-$(2\Sidx)$ perturbation).
\end{itemize}
The preconditioner is itself invertible by a parallel process, as will be explained next. We outline two methods to perform this inversion: a general one, which requires us to choose $\abs\alpha=1$, and one for the case $\tilde\Psi_\mathcal Q=O$, which can use any $\alpha\ne0$.

\paragraph{General method}
The matrix $C(\alpha)$ used in \cref{eq:diag:prec:Palpha} is an \emph{alpha-circulant} matrix -- that is, it is Toeplitz and each super-diagonal is equal to a value $\alpha$ times its complementing sub-diagonal. It is well-known (see\ukus{}{,} e.g.\ukus{}{,}\ \cite{biniNumericalMethodsStructured2005a}) that alpha-circulants diagonali\sz{}e as
\begin{equation} \label{eq:diag:prec:Calpha}
    C(\alpha) = VD(\alpha)V^{-1} \quad \text{with} \quad V=\Gamma_\alpha^{-1}\mathbb F^* \quad \text{and} \quad D(\alpha) = \mathrm{diag}(\sqrt{\widehat\Tidx}\mathbb F\Gamma_\alpha\mv c_1(\alpha))
\end{equation}
where $\mv c_1(\alpha)$ is $C(\alpha)$'s first column, $\mathbb F=\{\E^{2\pi\iu jk/\widehat\Tidx}/\sqrt{\widehat\Tidx}\}_{j,k=0}^{\widehat\Tidx-1}$ is the discrete Fourier matrix, and we define $\Gamma_\alpha = \diag(1, \alpha^{1/\widehat\Tidx}, \ldots, \alpha^{(\widehat\Tidx-1)/\widehat\Tidx})$. When $\abs\alpha=1$, it holds that $\Gamma_\alpha^{-1}=\Gamma_\alpha^*$ (where $\Gamma_\alpha^*$ is the Hermitian transpose of $\Gamma_\alpha$) and, therefore, $C(\alpha)$ and $C^*(\alpha)$ are simultaneously diagonali\sz{}able \cite{mezelfparadiag}. Then
\begin{equation*}
    P^{-1}(\alpha) = \Bigl(\begin{smallmatrix}
        \left[\begin{smallmatrix}
            \Gamma_\alpha^*\mathbb F^*\\&\Gamma_\alpha^*\mathbb F^*
        \end{smallmatrix}\right] \otimes I\end{smallmatrix}
    \Bigr)\left[\begin{smallmatrix}
        I\otimes I + D(\alpha)\otimes\tilde\Phi_\mathcal P & I \otimes \tilde\Psi_\mathcal P\\
        -I\otimes\tilde\Psi_\mathcal Q & I\otimes I + D^*(\alpha)\otimes\tilde\Phi_\mathcal Q
    \end{smallmatrix}\right]^{-1}\Bigl(\begin{smallmatrix}
        \left[\begin{smallmatrix}
            \mathbb F\Gamma_\alpha\\&\mathbb F\Gamma_\alpha
        \end{smallmatrix}\right] \otimes I\end{smallmatrix}
    \Bigr)    \mcom
\end{equation*}
where the inverted matrix on the right-hand side has only diagonal matrices as left operands in the Kronecker products. Thus inversion of $P(\alpha)$ can be decomposed into $\widehat\Tidx$ different inversions that can be solved in parallel, as implemented in \cref{alg:diag:prec:gen}. The Fourier matrix $\mathbb F$ can be applied efficiently (log-linearly in $\Tidx$) with the fast Fourier transform and can be paralleli\sz{}ed over the spatial dimensions of the problem.

\begin{algorithm}
    \caption{Procedure for inverting $P(\alpha)$ from \cref{eq:diag:prec:Palpha}} \label{alg:diag:prec:gen}
    \begin{tabular}{rl}
        \textbf{Input:}  &Vectors $\mv v$ and $\mv w$\\
        &Matrix $D(\alpha)$ following from the time discreti\sz{}ation by \cref{eq:diag:prec:Calpha} ($\hspace{.03cm}\abs\alpha=1$)\\
        &with diagonal elements $d_\tidx(\alpha)$\\
        \textbf{Output:} &The vector $\bigl[\begin{smallmatrix}\kronv x\\ \kronv z\end{smallmatrix}\bigr] = \kronm P^{-1}(\alpha)\bigl[\begin{smallmatrix}\kronv v\\ \kronv w\end{smallmatrix}\bigr]$\\
    \end{tabular}
    \begin{algorithmic}[1]
        \State Calculate $\kronv{r_1} \coloneqq (\mathbb F \Gamma_\alpha \kron I)\kronv{v}$, $\kronv{s_1} \coloneqq (\mathbb F \Gamma_\alpha \kron I)\kronv{w}$ using the (parallel) \textsc{fft}.
        \State For $\tidx=\{1, \ldots, \widehat\Tidx\}$, solve (in parallel)
            \begin{equation} \label{eq:alg:diag:prec:gen:sys}
                \begin{litmat}\mv r_{\mv{2,}\tidx}\\\mv s_{\mv{2,}\tidx}\end{litmat}\coloneqq
                \begin{litmat}
                    I + d_\tidx(\alpha)\tilde\Phi_\mathcal P & \tilde\Psi_\mathcal P\\
                    -\tilde\Psi_\mathcal Q & I + d_\tidx^*(\alpha)\tilde\Phi_\mathcal Q\\
                \end{litmat}^{-1}
                \begin{litmat}\mv r_{\mv{1,}\tidx}\\\mv s_{\mv{1,}\tidx}\end{litmat}    \mper
            \end{equation}
        \State Calculate $\kronv{x} = (\Gamma_\alpha^{-1}\mathbb F^* \kron I)\kronv{r_2}$, $\kronv{z} = (\Gamma_\alpha^{-1}\mathbb F^* \kron I)\kronv{s_2}$ using the (parallel) \textsc{fft}.
    \end{algorithmic}
\end{algorithm}

\paragraph{Method for a block-triangular preconditioner}
When $\tilde\Psi_\mathcal Q=O$ (as is often the case for terminal-cost objectives), we can invert the bottom and top halves of $P(\alpha)$ separately. Then simultaneous diagonali\sz{}ability of $C(\alpha)$ and $C^*(\alpha)$ is no longer needed and $\alpha$ can be any non-zero number (with small values generally working better \cite{mezelfparadiag}, as they decrease the difference between $P(\alpha)$ and $\tilde A$). Then it holds that
\begin{equation*}
    P^{-1}(\alpha) = \Bigl(\begin{smallmatrix}
        \left[\begin{smallmatrix}
            \Gamma_\alpha^{-1}\mathbb F^*\\&\Gamma_\alpha^*\mathbb F^*
        \end{smallmatrix}\right] \otimes I\end{smallmatrix}
    \Bigr)\left[\begin{smallmatrix}
        I\otimes I + D(\alpha)\otimes\tilde\Phi_\mathcal P & (\mathbb F\Gamma_\alpha\Gamma_\alpha^*\mathbb F^*) \otimes \tilde\Psi_\mathcal P\\
        & I\otimes I + D^*(\alpha)\otimes\tilde\Phi_\mathcal Q
    \end{smallmatrix}\right]^{-1}\Bigl(\begin{smallmatrix}
        \left[\begin{smallmatrix}
            \mathbb F\Gamma_\alpha\\&\mathbb F\Gamma_\alpha^{-*}
        \end{smallmatrix}\right] \otimes I\end{smallmatrix}
    \Bigr)    \mper
\end{equation*}
\Cref{alg:diag:prec:triangle} lays out how to multiply a vector by $P^{-1}(\alpha)$ when using this method.

\begin{algorithm}[h!]
    \caption{Procedure for inverting $P(\alpha)$ from \cref{eq:diag:prec:Palpha} when $\tilde\Psi_\mathcal Q = O$} \label{alg:diag:prec:triangle}
    \begin{tabular}{rl}
        \textbf{Input:}  &Vectors $\mv v$ and $\mv w$\\
        &Matrix $D(\alpha)$ following from the time discreti\sz{}ation by \cref{eq:diag:prec:Calpha} ($\hspace{.03cm}\abs\alpha\ne0$)\\
        &with diagonal elements $d_\tidx(\alpha)$\\
        \textbf{Output:} &The vector $\bigl[\begin{smallmatrix}\kronv x\\ \kronv z\end{smallmatrix}\bigr] = \kronm P^{-1}(\alpha)\bigl[\begin{smallmatrix}\kronv v\\ \kronv w\end{smallmatrix}\bigr]$\\
    \end{tabular}
    \begin{algorithmic}[1]
            \LineComment Phase 1: invert the bottom-right block
            \State Calculate $((\mv s_{\mv{1,}1})^\trsp, \ldots, (\mv s_{\mv{1,}\widehat\Tidx})^\trsp)^\trsp \coloneqq (\mathbb F \Gamma_\alpha^{-*} \kron I)\kronv{w}$ using the (parallel) \textsc{fft}.
            \State For $\tidx=\{1, \ldots, \widehat\Tidx\}$, solve (in parallel)
            \begin{equation}
                \mv s_{\mv{2,}\tidx} \coloneqq (I + d_\tidx^*(\alpha)\tilde\Phi_\mathcal Q)^{-1}\mv s_{\mv{1,}\tidx}
            \end{equation}
            and assemble $\kronv{s_2} \coloneqq ((\mv s_{\mv{2,}1})^\trsp, \ldots, (\mv s_{\mv{2,}\widehat\Tidx})^\trsp)^\trsp$.
            \State Calculate $\kronv{z} = (\Gamma_\alpha^*\mathbb F ^* \kron I)\kronv{s_2}$ using the (parallel) \textsc{fft}.
            \vspace{.3cm}
            \LineComment Phase 2: invert the rest of the matrix
            \State Set $\kronv{r_1} = \kronv v - (I \kron \tilde\Psi_\mathcal P)\kronv z$.
            \State Calculate $((\mv r_{\mv{2,}1})^\trsp, \ldots, (\mv r_{\mv{2,}\widehat\Tidx})^\trsp)^\trsp \coloneqq (\mathbb F \Gamma_\alpha \kron I)\kronv{r_1}$ using the (parallel) \textsc{fft}.
            \State For $\tidx=\{1, \ldots, \widehat\Tidx\}$, solve (in parallel)
            \begin{equation}
                \mv r_{\mv{3,}\tidx} \coloneqq (I + d_\tidx(\alpha)\tilde\Phi_\mathcal P)^{-1}\mv r_{\mv{2,}\tidx}
            \end{equation}
            and assemble $\kronv{r_3} \coloneqq ((\mv r_{\mv{3,}1})^\trsp, \ldots, (\mv r_{\mv{3,}\widehat\Tidx})^\trsp)^\trsp$.
            \State Calculate $\kronv{x} = (\Gamma_\alpha^{-1}\mathbb F ^* \kron I)\kronv{r_3}$ using the (parallel) \textsc{fft}.
    \end{algorithmic}
\end{algorithm}

%% file: src/diag/small.tex
In \cref{alg:diag:prec:gen,alg:diag:prec:triangle}, it is needed to solve $\widehat\Tidx$ linear systems in parallel to each other. In the general method (\cref{alg:diag:prec:gen}), these systems \cref{eq:alg:diag:prec:gen:sys} use matrices of the form
\begin{equation}
    H_\tidx \coloneqq \begin{litmat}
        I + d_\tidx(\alpha)\tilde\Phi_\mathcal P & \tilde\Psi_\mathcal P\\
        -\tilde\Psi_\mathcal Q & I + d_\tidx^*(\alpha)\tilde\Phi_\mathcal Q
    \end{litmat}
\end{equation}
and this subsection outlines how $H_\tidx$ can be inverted. This is non-trivial due to the matrices $\tilde\Phi_\mathcal P$, $\tilde\Phi_\mathcal Q$, $\tilde\Psi_\mathcal P$, and $\tilde\Psi_\mathcal Q$, which are defined through the coarse propagators \cref{eq:conv:setting:lin:Ptilde,eq:conv:setting:lin:Qtilde}. We focus on the general method; the speciali\sz{}ed case of \cref{alg:diag:prec:triangle} is simpler and can be treated analogously. We propose two methods for solving systems with $H_\tidx$, differing both in generality and in performance.

\clearpage
\paragraph{Method 1: Black-box approach}
One strength of the ParaOpt algorithm is that the propagators can be given as black boxes (although for our preconditioners, we do require the coarse ones to be affine). In that case we can only access the $\tilde\Phi$ and $\tilde\Psi$ matrices through the propagators $\tilde{\mathcal P}$ and $\tilde{\mathcal Q}$ by \cref{eq:conv:setting:lin:Ptilde,eq:conv:setting:lin:Qtilde}. Without giving up on the black-box character of the coarse solvers, the smaller systems can be tackled with an iterative solver, which only needs $H_\tidx$ as a multiplication routine. It can be seen that, for any $\mv x$ and $\mv z$,
\begin{equation}
    H_\tidx\begin{litmat}\mv x\\\mv z\end{litmat} = \begin{litmat}
        \mv x + \tilde{\mathcal P}(d_\tidx(\alpha)\mv x, -\mv z) - \tilde{\mathcal P}(\mv0, \mv0)\\
        \mv z + \tilde{\mathcal Q}(-\mv x, d_\tidx^*(\alpha)\mv z) - \tilde{\mathcal Q}(\mv0, \mv0)
    \end{litmat}
\end{equation}
where the quantities $\tilde{\mathcal P}(\mv0, \mv0)$ and $\tilde{\mathcal Q}(\mv0, \mv0)$ can be precomputed.

\paragraph{Method 2: Using the coarse propagators' explicit form}
In many cases, the coarse propagators will be simple and their explicit form known. Then it is often possible to solve the linear system more cheaply than with the black-box approach above. \Cref{ex:diag:small:tr-ie1} illustrates this for the case of a simple coarse propagator.
\begin{example}[Tracking with one-step implicit Euler] \label{ex:diag:small:tr-ie1}
    ParaOpt with a tracking objective and a one-step implicit-Euler coarse propagator for linear problems has that
    \begin{equation}
        \tilde\Phi = (I+\DT K)^{-1} \quad \text{and} \quad \tilde\Psi = \widehat\gamma(I + \DT K)^{-1}    \mcom
    \end{equation}
    as can be derived using \cref{sec:apdx-po-prop:phipsi}'s arguments. Write $Z\coloneqq I+\DT K$. Then
    \begin{equation}
        \Bigl[\begin{smallmatrix}
            I + d_l(\alpha)Z^{-1} & \widehat\gamma Z^{-1}\\
            -\widehat\gamma Z^{-1} & I + d_l^*(\alpha)Z^{-1}\\
        \end{smallmatrix}\Bigr]\Bigl[\begin{smallmatrix}
            \mv{r_2}\\\mv{s_2}
        \end{smallmatrix}\Bigr]=\Bigl[\begin{smallmatrix}
            \mv{r_1}\\\mv{s_1}
        \end{smallmatrix}\Bigr]
        \Leftrightarrow \Bigl[\begin{smallmatrix}
            Z + d_l(\alpha)I & \widehat\gamma I\\
            -\widehat\gamma I & Z + d_l^*(\alpha)I\\
        \end{smallmatrix}\Bigr]\Bigl[\begin{smallmatrix}
            \mv{r_2}\\\mv{s_2}
        \end{smallmatrix}\Bigr]=\Bigl[\begin{smallmatrix}
            Z\mv{r_1}\\Z\mv{s_1}
        \end{smallmatrix}\Bigr]    \mper
    \end{equation}
    The second form is much easier to solve, by either direct or iterative methods.
\end{example}

Both approaches above can use iterative solvers, while only method 2 can use direct methods. In the context of \textsc{ivp} ParaDiag algorithms, multiple techniques have been proposed to accelerate solving related linear systems \cite{heVankatypeMultigridSolver2022a,liuROMacceleratedParallelintimePreconditioner2020}. Adaptations of those methods could conceivably further improve the efficiency of our preconditioner.

%% file: src/diag/conv.tex
To assess the convergence properties of solving systems with $\tilde A$ using the proposed preconditioners $P(\alpha)$, we take the usual approach of studying the eigenvalues of the preconditioned system matrix $P^{-1}(\alpha)\tilde A$. If those are clustered together and lie far enough away from zero, rapid convergence is expected for most iterative linear-system solvers such as \textsc{gmres} \cite{saadGMRESGeneralizedMinimal1986a}.

An advantage of using ParaDiag-inspired preconditioners is that, in certain cases, eigenvalue results from ParaDiag apply directly. In \cite{mezelfparadiag}, analytic expressions are provided for the preconditioned eigenvalues of optimi\sz{}ation ParaDiag methods applied to linear diffusive problems, for both objective functions we study. When \cref{ass:conv:setting:1,ass:conv:setting:2} are satisfied \emph{for the coarse propagator}, those eigenvalue expressions also apply to the proposed ParaOpt preconditioners. In particular, for tracking and when choosing $\alpha=-1$ in \cref{eq:diag:prec:Palpha}, it is shown in \cite{mezelfparadiag} that \textsc{gmres} converges exponentially with a problem-independent convergence rate under \cref{ass:conv:setting:1,ass:conv:setting:2,ass:conv:setting:3}.

Even in the most general situation, where nothing is known about the coarse propagators other than them being affine, preconditioned coarse-grid correction asymptotically scales well with increasing time-parallelism, as the following theorem asserts.
\begin{theorem}
    The preconditioned matrix $P^{-1}(\alpha)\tilde A$ has at most $2\Sidx$ eigenvalues that differ from $1$.
\end{theorem}
\begin{proof}
    The difference $\tilde A-P(\alpha)$ between the matrices has maximum rank $2\Sidx$. It is a well-known result (mentioned in\ukus{}{,} e.g.\ukus{}{,}\ \cite{wuParallelInTimeBlockCirculantPreconditioner2020a}) that this proves the theorem.
\end{proof}

%% file: src/proof-tr/proof-tr.tex
Given the discussion in \cref{sec:conv:setting} it holds that $\rho=\max_{\sigma\in\mathrm{eig}(K)}\max(\kern2pt\abs{\mathrm{eig}(S_\sigma)})$, with $S_\sigma$ given by \cref{eq:conv:setting:Ssigma:tr}. Define $B$ and $\tilde B$ such that
\begin{equation} \label{eq:proof-tr:proof-tr:Ssigma}
    S_\sigma = I - \begin{litmat}
        \tilde B & \tilde\psi I\\
        -\tilde\psi I & \tilde B^\trsp\\
    \end{litmat}^{-1}\begin{litmat}
        B & \psi I\\
        -\psi I & B^\trsp\\
    \end{litmat}    \mper
\end{equation}

We will first prove a general result about the eigenvalues $(1-\theta)$ of \cref{eq:proof-tr:proof-tr:Ssigma}, making abstraction of the forms of $B$ and $\tilde B$.
\begin{lemma} \label{lmm:proof-tr:1-theta2}
    Let $\theta$ be an eigenvalue of
    \begin{equation}
        \begin{litmat}
            \tilde B & \tilde \psi I\\
            -\tilde \psi I & \tilde B^\trsp\\
        \end{litmat}^{-1}\begin{litmat}
            B & \psi I\\
            -\psi I & B^\trsp\\
        \end{litmat}
    \end{equation}
    with $B,\tilde B \in \mathbb R^{\widehat\Tidx\times\widehat\Tidx}$. Then, for some vector $\mv v \in \mathbb C^{\widehat\Tidx}$,
    \begin{equation} \label{lmm:proof-tr:1-theta2:1-theta2}
        \abs{1 - \theta}^2 = \frac{\mv v^*((B-\tilde B)^\trsp(B-\tilde B)+(\psi-\tilde\psi)^2I)\mv v}{\mv v^*(\tilde B^\trsp \tilde B+\tilde\psi^2I)\mv v}    \mper
    \end{equation}
\end{lemma}
\begin{proof}
    The proof of this lemma is inspired by \cite{wuParallelInTimeBlockCirculantPreconditioner2020a}, which in turn refers to \cite{simonciniSpectralPropertiesHermitian2004}. An eigenvalue $\theta$ and its corresponding eigenvector $(\mv v^\trsp, \mv w^\trsp)^\trsp$ satisfy
    \begin{equation}
        \begin{litmat}
            B & \psi I\\
            -\psi I & B^\trsp\\
        \end{litmat}\begin{litmat}
            \mv v\\ \mv w\\
        \end{litmat} = \theta\begin{litmat}
            \tilde B & \tilde \psi I\\
            -\tilde \psi I & \tilde B^\trsp\\
        \end{litmat}\begin{litmat}
            \mv v\\ \mv w\\
        \end{litmat}    \mcom
    \end{equation}
    which is equivalent to the system
    \begin{subequations} \label{eq:paraopt:eig}
    \begin{align}
        \label{eq:paraopt:eig2} B\mv v + \psi\mv w &= \theta(\tilde B\mv v + \tilde\psi\mv w)    \mcom\\
        \label{eq:paraopt:eig1} -\psi\mv v + B^\trsp\mv w &= \theta(-\tilde\psi\mv v + \tilde B^\trsp \mv w)    \mper
    \end{align}
    \end{subequations}

    Then from (\ref{eq:paraopt:eig2}) follows
    \begin{equation}
        \mv w = \frac{ \theta \tilde B - B}{\psi-\theta \tilde\psi}\mv v    \mcom
    \end{equation}
    which, when filled into (\ref{eq:paraopt:eig1}), yields
    \begin{equation}
        -\psi\mv v + B^\trsp\frac{ \theta \tilde B - B}{\psi-\theta \tilde\psi}\mv v = \theta(\tilde B\mv v + \tilde\psi\frac{ \theta \tilde B - B}{\psi-\theta \tilde\psi}\mv v)    \mper
    \end{equation}
    After left-multiplying by $\mv v^*$, this can be manipulated into a quadratic equation in $\theta$:
    \begin{equation} \label{eq:paraopt:quad}
    \begin{aligned}
        a\theta^2 - b\theta + c \coloneqq& \mv v^*(\tilde B^\trsp \tilde B + \tilde\psi^2I)\mv v \theta^2 \\{}-{}& \mv v^*(B^\trsp \tilde B + 2\psi\tilde\psi I+\tilde B^\trsp B)\mv v \theta \\{}+{}& \mv v^*(B^\trsp B + \psi^2I)\mv v = 0    \mper
    \end{aligned}
    \end{equation}
    Note that $a$, $b$, and $c$ are real numbers due to the matrices between parentheses being real and symmetric.
    The solutions to (\ref{eq:paraopt:quad}) are found as
    \begin{equation}
        \theta_\pm = \frac{b}{2a} \pm \sqrt{\left(\frac{b}{2a}\right)^2 - \frac c a} \eqqcolon \Re(\theta) \pm \iu\Im(\theta)    \mper
    \end{equation}
    Later, \cref{lmm:paraopt:b^2-4ac} will prove that $\left(\frac{b}{2a}\right)^2 - \frac c a \le 0$ always holds. We can then state $\Re(\theta)^2 = \left(\frac{b}{2a}\right)^2$ and $\Im(\theta)^2 = \frac c a - \left(\frac{b}{2a}\right)^2$. It holds that
    \begin{equation} \label{eq:paraopt:conv:compltheta}
    \begin{aligned}
        \abs{1 - \theta}^2 &= \Re(1 - \theta)^2 + \Im(1 - \theta)^2 = (1 - \Re(\theta))^2 + \Im(\theta)^2 \\
        &= \left(1 - \frac{b}{2a}\right)^2 + \left(\frac c a - \left(\frac{b}{2a}\right)^2\right) = 1 + \frac c a - \frac b a\\
        &= \frac{\mv v^*(B^\trsp B+\tilde B^\trsp \tilde B-B^\trsp \tilde B-\tilde B^\trsp B+\psi^2I+\tilde\psi^2I-2\psi\tilde\psi I)\mv v}{\mv v^*(\tilde B^\trsp \tilde B+\tilde\psi^2I)\mv v}\\
        &= \frac{\mv v^*((B-\tilde B)^\trsp(B-\tilde B)+(\psi-\tilde\psi)^2I)\mv v}{\mv v^*(\tilde B^\trsp \tilde B+\tilde\psi^2I)\mv v}    \mcom
    \end{aligned}
    \end{equation}
    which is exactly (\ref{lmm:proof-tr:1-theta2:1-theta2}).
\end{proof}
\begin{lemma} \label{lmm:paraopt:b^2-4ac}
    In the quadratic equation (\ref{eq:paraopt:quad}), it always holds that
    \begin{equation} \label{eq:lmm:paraopt:b^2-4ac}
        b^2-4ac \le 0    \mper
    \end{equation}
\end{lemma}
\begin{proof}
    We normali\sz{}e $\mv v$ without loss of generality. By defining $\mv v_1 \coloneqq \tilde B^\trsp\mv v$ and $\mv v_2 \coloneqq B^\trsp\mv v$, we can write
    \begin{equation}
        a = \mv v_1^*\mv v_1 + \tilde\psi^2, \quad b = 2\mv v_2^*\mv v_1 + 2\psi\tilde\psi, \quad \text{and} \quad c = \mv v_2^*\mv v_2 + \psi^2    \mper
    \end{equation}
    Define $\mv{\widehat v_1}$ by appending $\tilde\psi$ to $\mv v_1$ and $\mv{\widehat v_2}$ by doing the same with $\psi$ and $\mv v_2$. Then the expression $b^2-4ac$ is equal to
    \begin{equation}
        4\dualp{\mv{\widehat v_1}}{\mv{\widehat v_2}}^2 - 4\dualp{\mv{\widehat v_1}}{\mv{\widehat v_1}}\dualp{\mv{\widehat v_2}}{\mv{\widehat v_2}} = 4(\cos^2(\omega)-1)\dualp{\mv{\widehat v_1}}{\mv{\widehat v_1}}\dualp{\mv{\widehat v_2}}{\mv{\widehat v_2}}    \mcom
    \end{equation}
    with $\omega$ the angle between $\mv{\widehat v_1}$ and $\mv{\widehat v_2}$. Since $-1 \le \cos(\omega)\le 1$, the lemma holds.
\end{proof}

Having proven the general \cref{lmm:proof-tr:1-theta2}, we can fill in our particular $B$ and $\tilde B$ matrices. If $(1-\theta)$ is an eigenvalue of $S_\sigma$, it holds that, for some $\mv v$,
\begin{equation}
    \abs{1-\theta}^2 = \frac{\mv v^*M_1\mv v + \norm{\mv v}_2^2(\psi-\tilde\psi)^2}{\mv v^*M_2\mv v + \norm{\mv v}_2^2\tilde\psi^2}
\end{equation}
with
\begin{equation} \label{eq:proof-tr:proof-tr:Ms}
    M_1 = \left[\begin{smallmatrix}
        (\tilde\varphi-\varphi)^2\\&\ddots\\&&(\tilde\varphi-\varphi)^2\\&&& 0
    \end{smallmatrix}\right] \quad \text{and} \quad M_2 = \left[\begin{smallmatrix}
        1+\tilde\varphi^2 & -\tilde\varphi\\
        -\tilde\varphi & \ddots & \ddots\\
        & \ddots & 1+\tilde\varphi^2 & -\tilde\varphi\\
        &&-\tilde\varphi & 1\\
    \end{smallmatrix}\right]    \mper
\end{equation}
We can normali\sz{}e $\mv v$ such that $\norm{\mv v}_2 = 1$ without altering the value of $\abs{1-\theta}$. Then
\begin{equation} \label{eq:proof-tr:proof-tr:bound1-theta}
\begin{aligned}
    \abs{1-\theta}^2 = \frac{\mv v^*M_1\mv v + (\psi-\tilde\psi)^2}{\mv v^*M_2\mv v + \tilde\psi^2}
      &\le \frac{(\tilde\varphi-\varphi)^2 + (\psi-\tilde\psi)^2}{\mv v^*M_2\mv v + \tilde\psi^2}
      \\&< \frac{(\tilde\varphi-\varphi)^2 + (\psi-\tilde\psi)^2}{(1-\tilde\varphi)^2 + \tilde\psi^2}    \mper
\end{aligned}
\end{equation}
The first inequality is valid since $\mv v^*M_1\mv v = (1-\abs{v_{\widehat\Tidx}}^2)(\tilde\varphi-\varphi)^2 \le (\tilde\varphi-\varphi)^2$, where $0 \le \abs{v_{\widehat\Tidx}} \le 1$ is the magnitude of the last element in $\mv v$. The second inequality is more involved, and is proven by the following lemma. Then, \cref{eq:proof-tr:proof-tr:bound1-theta} proves \cref{thm:conv:conv:tr-gen}.
\begin{lemma} \label{lmm:paraopt:M2eig}
    Any eigenvalue $\xi$ of $M_2$, defined in \cref{eq:proof-tr:proof-tr:Ms}, satisfies $\xi > (1-\tilde\varphi)^2$.
\end{lemma}
\begin{proof}
    From \cite[Theorem 4]{betterThanYueh}, one can deduce that the matrix
    \begin{equation}
        \widehat M_2 \coloneqq \left[\begin{smallmatrix}
            1+\tilde\varphi^2 & -\tilde\varphi\\
            -\tilde\varphi & 1+\tilde\varphi^2 & \ddots\\
            & \ddots & \ddots & -\tilde\varphi\\
            &&-\tilde\varphi & 1-\tilde\varphi+\tilde\varphi^2\\
        \end{smallmatrix}\right]
    \end{equation}
    has eigenvalues
    \begin{equation} \label{eq:paraopt:M2eig:xihat}
        \widehat\xi_j = 1 + \tilde\varphi^2 + 2\tilde\varphi\cos\frac{2j\pi}{2\widehat\Tidx+1}, \quad j = 1, \ldots, \widehat\Tidx    \mcom
    \end{equation}
    which means $\widehat\xi_j > (1-\tilde\varphi)^2$. To transform $\widehat M_2$ into $M_2$, one adds $(\tilde\varphi - \tilde\varphi^2)$ to the last diagonal element. Since this is a positive number (\cref{ass:conv:setting:3} ensures that $0<\tilde\varphi<1$), it cannot reduce the minimum eigenvalue of this symmetric matrix (as follows from \cite[Theorem 10.3.1]{parlettSymmetricEigenvalueProblem1998}). This proves the lemma.
\end{proof}

%% file: src/proof-tc/proof-tc.tex
Given the discussion in \cref{sec:conv:setting} it holds that $\rho=\max_{\sigma\in\mathrm{eig}(K)}\max(\kern2pt\abs{\mathrm{eig}(S_\sigma)})$, with $S_\sigma$ given by \cref{eq:conv:setting:Ssigma:tc}. Define $B$ and $\tilde B$ such that
\begin{equation} \label{eq:proof-tc:proof-tc:Ssigma}
    S_\sigma = I - \begin{litmat}
        \tilde B & \tilde\psi I\\
        -E & \tilde B^\trsp\\
    \end{litmat}^{-1}\begin{litmat}
        B & \psi I\\
        -E & B^\trsp\\
    \end{litmat}    \mcom
\end{equation}
where $E$ is all-zero except for a one in the bottom-right corner. This is analogous to the proof in \cref{sec:proof-tr}. However, we have no equivalent of \cref{lmm:paraopt:b^2-4ac}, so we will need to deal with separate cases for non-real and real eigenvalues of \cref{eq:proof-tc:proof-tc:Ssigma}.

\paragraph{Non-real eigenvalues}
We start by looking at the non-real case.

\begin{lemma} \label{lmm:proof-tc:1-theta2}
    Let $\theta$ be an eigenvalue with non-zero imaginary part of
    \begin{equation}
        \begin{litmat}
            \tilde B & \tilde \psi I\\
            -E & \tilde B^\trsp\\
        \end{litmat}^{-1}\begin{litmat}
            B & \psi I\\
            -E & B^\trsp\\
        \end{litmat}
    \end{equation}
    with $B,\tilde B \in \mathbb R^{\widehat\Tidx\times\widehat\Tidx}$ and where $E$ is as in \cref{eq:proof-tc:proof-tc:Ssigma}. Then, for some vector $\mv v \in \mathbb C^{\widehat\Tidx}$,
    \begin{equation}
        \abs{1 - \theta}^2 = \frac{\mv v^*((B-\tilde B)(B-\tilde B)^\trsp)\mv v}{\mv v^*(\tilde B \tilde B^\trsp+\tilde\psi E)\mv v}    \mper
    \end{equation}
\end{lemma}
\begin{proof}
    The proof is similar to that of \cref{lmm:proof-tr:1-theta2}. We introduce an eigenvalue $\theta$ and its corresponding eigenvector $(\mv v^\trsp, \mv w^\trsp)^\trsp$, such that
    \begin{subequations} \label{eq:proof-tc:proof-tc:eig}
    \begin{align}
        \label{eq:proof-tc:proof-tc:eig:1} B\mv v + \psi\mv w &= \theta(\tilde B\mv v + \tilde\psi\mv w)    \mcom\\
        \label{eq:proof-tc:proof-tc:eig:2} -E\mv v + B^\trsp\mv w &= \theta(-E\mv v + \tilde B^\trsp \mv w)    \mcom
    \end{align}
    \end{subequations}
    and thus $\mv w = \frac{\theta\tilde B - B}{\psi-\theta\tilde\psi}\mv v$. After filling this into \cref{eq:proof-tc:proof-tc:eig:2} and left-multiplying by $\mv v^*$, we can manipulate everything into the quadratic equation
    \begin{equation}
    \begin{aligned}
        a\theta^2 - b\theta + c \coloneqq& \mv v^*(\tilde B^\trsp \tilde B + \tilde\psi E)\mv v \theta^2 \\{}-{}& \mv v^*(B^\trsp \tilde B + (\psi+\tilde\psi)E + \tilde B^\trsp B)\mv v \theta \\{}+{}& \mv v^*(B^\trsp B + \psi E)\mv v = 0
    \end{aligned}
    \end{equation}
    with solution $\theta_\pm = \frac{b}{2a} \pm \sqrt{\left(\frac{b}{2a}\right)^2 - \frac c a}$.
    \clearpage

    Since we look for solutions $\theta$ with non-zero imaginary part, the contents of the square root must be negative. Then, similarly to before,
    \begin{equation}
        \abs{1-\theta}^2 = 1+\frac ca - \frac ba = \frac{\mv v^*((B-\tilde B)^\trsp(B-\tilde B))\mv v}{\mv v^*(\tilde B^\trsp\tilde B + \tilde\psi E)\mv v}    \mper
    \end{equation}
\end{proof}

We utili\sz{}e this lemma in the same way as before, but now must keep in mind that its result is only guaranteed for non-real eigenvalues. In that case, we can state
\begin{equation} \label{eq:proof-tc:proof-tc:complex}
    \abs{1-\theta} = \sqrt\frac{\mv v^*M_1\mv v}{\mv v^*M_2\mv v + \tilde\psi\abs{v_{\widehat\Tidx}}^2} \le \frac{\abs{\varphi-\tilde\varphi}}{1-\tilde\varphi}
\end{equation}
where we recall $M_1$ and $M_2$ from \cref{eq:proof-tr:proof-tr:Ms}.

\paragraph{Real eigenvalues}
When the contents of the square root in \cref{eq:proof-tc:proof-tc:eig}'s solution are non-negative, the reasoning breaks down. This case can be treated in a different way. By left-multiplying \cref{eq:proof-tc:proof-tc:eig:1,eq:proof-tc:proof-tc:eig:2} by $\mv w^*$ and $\mv v^*$, respectively, and then subtracting the complex adjoint of the latter from the former, we find that
\begin{equation}
    \mv w^*\psi\mv w + \mv v^*E\mv v = \theta(\mv w^*\tilde\psi\mv w + \mv v^*E\mv v) \Leftrightarrow \theta = \frac{\psi\norm{\mv w}_2^2 + \abs{v_{\widehat\Tidx}}^2}{\tilde\psi\norm{\mv w}_2^2 + \abs{v_{\widehat\Tidx}}^2}    \mcom
\end{equation}
such that, shifting our focus to $(1-\theta)$, we obtain
\begin{equation}
    (1-\theta) = \frac{(\tilde\psi-\psi)\norm{\mv w}_2^2}{\tilde\psi\norm{\mv w}_2^2 + \abs{v_{\widehat\Tidx}}^2}    \mper
\end{equation}
As shown in \cite[(3.24) and (3.28)]{ganderPARAOPTPararealAlgorithm2020a}, it holds that $v_{\widehat\Tidx} = w_{\widehat\Tidx}$ and $w_\tidx = w_{\widehat\Tidx}(\tilde\varphi+\frac{\varphi-\tilde\varphi}{1-\theta})^{\widehat\Tidx-\tidx}$. A normali\sz{}ation such that $w_{\widehat\Tidx}=1$ and the definition $x \coloneqq (1 - \theta)$ then yield
\begin{equation} \label{eq:proof-tc:proof-tc:sumfrac}
    x = \frac{\tilde\psi-\psi}{\tilde\psi + 1/\sum_{\tidx=0}^{\widehat\Tidx-1}(\tilde\varphi+\frac{\varphi-\tilde\varphi}{x})^{2\tidx}} \Leftrightarrow f_{\widehat\Tidx}(x) \coloneqq \frac{\tilde\psi-\psi}{\tilde\psi + 1/\sum_{\tidx=0}^{\widehat\Tidx-1}(\tilde\varphi+\frac{\varphi-\tilde\varphi}{x})^{2\tidx}} - x = 0 \mper
\end{equation}
Define $g(x) \coloneqq (\tilde\psi-\psi)/\tilde\psi - x$. We have, depending on $\mathrm{sign}(\tilde\psi-\psi)$,
\begin{equation}
    \forall x: f_{\widehat\Tidx}(x) \le f_\infty(x) \le g(x) \quad \text{or} \quad \forall x: f_{\widehat\Tidx}(x) \ge f_\infty(x) \ge g(x)    \mper
\end{equation}
Now denote by $x_1$ a root of $f_{\widehat\Tidx}$, and by $x_2$ the root of $g$. Clearly,
\begin{equation} \label{eq:proof-tc:proof-tc:x12}
    \mathrm{sign}(x_1) = \mathrm{sign}(x_2) \quad \text{and} \quad \abs{x_1} \le \abs{x_2}    \mper
\end{equation}
This means that $\min\{f_\infty(x_1), f_\infty(x_2)\} \le 0 \le \max\{f_\infty(x_1), f_\infty(x_2)\}$. Since $f_{\widehat\Tidx}$, $f_\infty$, and $g$ are continuous, Bolzano's theorem asserts that
\begin{equation}
    f_\infty(x^*) = 0 \quad \text{for some} \quad \min\{x_1, x_2\} \le x^* \le \max\{x_1, x_2\}    \mcom
\end{equation}
which, by \cref{eq:proof-tc:proof-tc:x12}, means that $\abs{x^*} \ge \abs{x_1}$. Recall that any \emph{real} eigenvalue $x = (1-\theta)$ of the ParaOpt iteration matrix must be a root of $f_{\widehat\Tidx}$; by the above argument, then, $\abs{1-\theta}$ is bounded from above by the absolute value of at least one root of
\begin{equation}
    f_\infty(x) = \frac{\tilde\psi-\psi}{\tilde\psi + 1/\sum_{\tidx=0}^{\infty}(\tilde\varphi+\frac{\varphi-\tilde\varphi}{x})^{2\tidx}} - x    \mper
\end{equation}
This root is efficiently computable: if $x=(\tilde\psi-\psi)/\tilde\psi$ makes the infinite sum diverge, it is a root (and the one with the largest absolute value). Otherwise, all roots must have the sum converge, which can then be replaced by $(1-(\tilde\varphi+\frac{\varphi-\tilde\varphi}{x})^{2})^{-1}$. Then finding $f_\infty$'s roots amounts to solving a quadratic equation and checking when the sum converges. Our numerical tests suggest that $f_\infty$ always has exactly one root.

%% file: src/num/intro.tex
\Cref{sec:conv} presents upper bounds on the ParaOpt convergence factor for linear diffusive problems. \Cref{sec:num:conv-tr,sec:num:conv-tc} will study the accuracy and sharpness of these bounds through numerical tests. Later, \cref{sec:num:prec} looks at the performance of the proposed preconditioners.

Our tests confirm the accuracy of our bounds and the scalability of preconditioned ParaOpt. To perform them, we have extended the \texttt{pintopt} package\footnote{The version of the \texttt{pintopt} \textsc{Matlab} package used here and code to reproduce our results are located at \url{https://gitlab.kuleuven.be/numa/public/pintopt}. To keep this paper reproducible, new additions and bugfixes are tracked only at \url{https://github.com/ArneBouillon/pintopt}.} designed in \cite{mezelfparadiag} to include a sequential implementation of ParaOpt and its preconditioners. Though unoptimi\sz{}ed, it is a useful reference solver and can be used to study iteration counts.

%% file: src/num/conv-tr.tex
Consider the scalar equation
\begin{equation}
    y'(t) = -\sigma y(t) + u(t), \quad y(0) = y_\mathrm{init}
\end{equation}
with a tracking objective function, analogously to a test case in \cite{ganderPARAOPTPararealAlgorithm2020a}. We follow \cite{ganderPARAOPTPararealAlgorithm2020a} in setting $\gamma=1$, $\sigma=16$ and $T=1$ and compare the true spectral radius $\rho$ of ParaOpt's iteration matrix \cref{eq:conv:setting:Ssigma} to the bound $\rho^*$ set by \cref{eq:thm:conv:conv:tr-gen:bound}.

Figure \ref{fig:po-num:scalar:vartimestep} uses $\widehat\Tidx=100$ and looks at the influence of the time steps. In the left figure, the coarse step $\Dt=\DT$ is kept fixed and the fine one $\dt$ is varied. The bound $\rho^*$ is correct and rather sharp, and $\rho$ becomes steady since $\dt\rightarrow0$ corresponds to the limit of an exact solver. The right figure fixes the fine time step $\dt=10^{-5}\DT$ and varies the coarse one. Expectedly, a small coarse time step leads to faster convergence.

To confirm that this spectral radius has the impact on convergence expected from \cref{sec:conv:setting}, ParaOpt has been executed on two scalar problems. Both use $\Tidx=50$ and $T=50$, with $y_\mathrm{init} = y_\mathrm d(\cdot) = 1$. The fine propagator is exact while the coarse one uses 10 steps of implicit Euler. The parameter sets, called A and B, are displayed in \cref{fig:po-num:scalar:conv-ill:1}: case A uses $\widehat\sigma=10^{-6}$ and $\widehat\gamma=6$, while case B uses $\widehat\sigma=0.0006$ and $\widehat\gamma=0.4$. They are overlaid on a copy of \cref{fig:conv:interp:rhostar:tr-10}, which shows $\rho^*$ for these propagators. Parameters B lead to a smaller $\rho^*$, so ParaOpt may be expected to converge faster for that problem than for A; \cref{fig:po-num:scalar:conv-ill:2} confirms this. In addition, the residuals decrease at a rate close to the bound, confirming its accuracy.

Next, consider the weak-scaling regimes from \cref{sec:diag:scale}. We use an exact fine propagator and a one-step implicit-Euler coarse one. Recall that the upper bound is $\widehat\Tidx$-independent, which guarantees weak scalability for fixed $\DT$. \Cref{fig:po-num:scalar:scaling:fixedDT} confirms the bound is sharp. For fixed $T$ (i.e.\ukus{}{,}\ decreasing $\DT$), \cref{fig:po-num:scalar:scaling:fixedT} shows good scalability.

\begin{figure}
    \centering
    \begin{subfigure}[c]{.38\textwidth}
        \includegraphics[width=\textwidth]{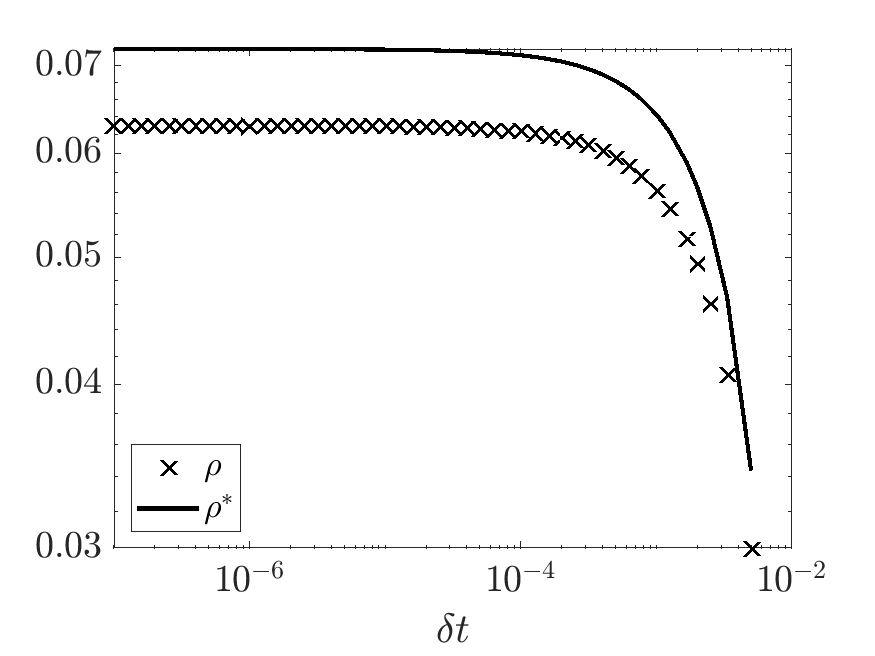}
    \end{subfigure}
    \hfill
    \begin{subfigure}[c]{.38\textwidth}
        \includegraphics[width=\textwidth]{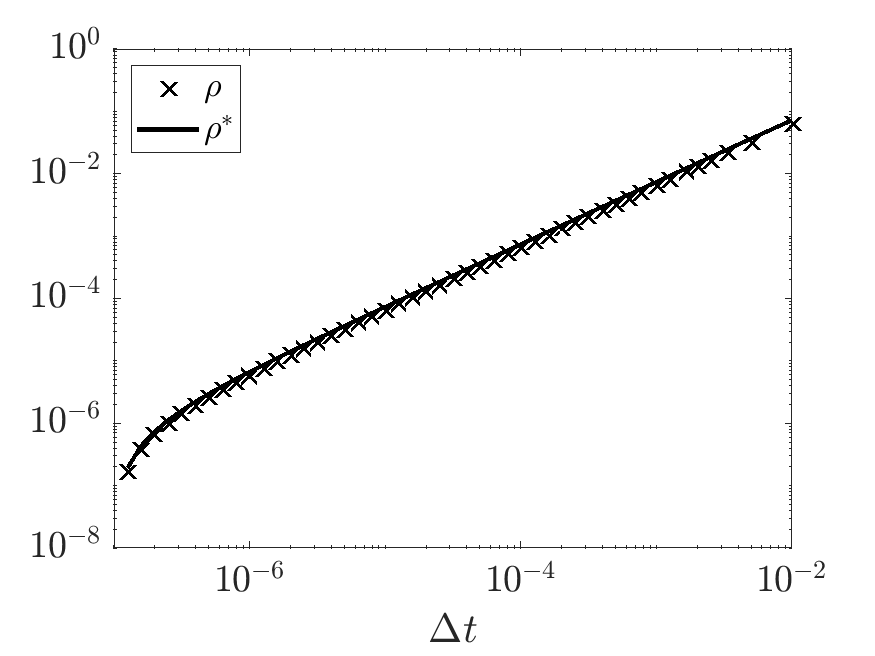}
    \end{subfigure}
    \caption{
        Spectral radius $\rho$ and bound $\rho^*$ from \cref{eq:thm:conv:conv:tr-gen:bound} of the ParaOpt iteration matrix for a fixed coarse time step $\Dt=\DT$ (left) and a fixed fine time step $\dt = 10^{-5}\DT$ (right), for the tracking problem with implicit-Euler propagators from \cref{sec:num:conv-tr}
        \vspace{-.9cm}
    }
    \label{fig:po-num:scalar:vartimestep}
\end{figure}

\begin{figure}
    \centering
    \begin{subfigure}[b!]{.42\textwidth}
        \includegraphics[width=\textwidth]{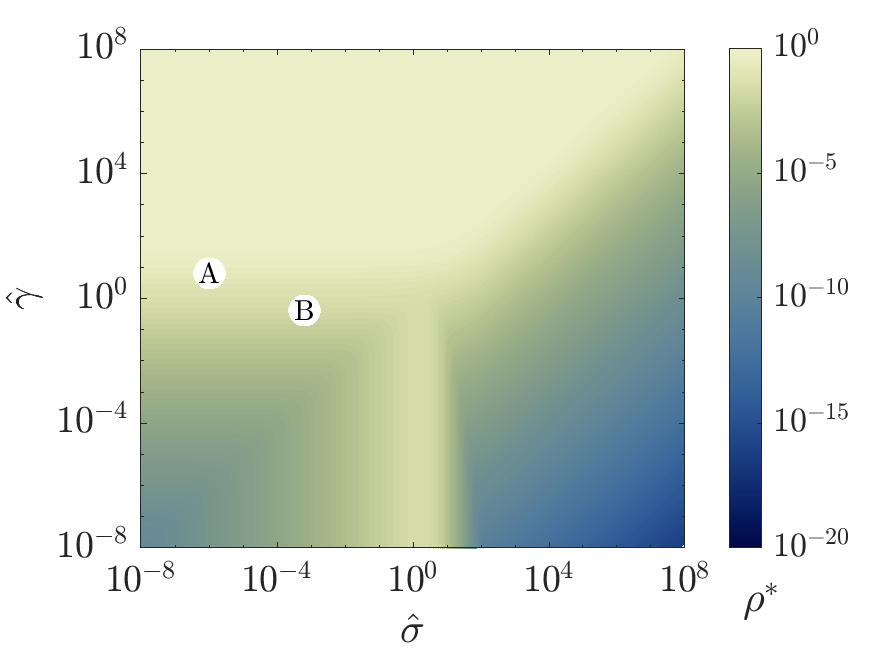}
        \caption{
            $\rho^*$ for parameter choices A and B
            \vspace{-.6cm}
        }
        \label{fig:po-num:scalar:conv-ill:1}
    \end{subfigure}
    \hfill
    \begin{subfigure}[b!]{.42\textwidth}
        \includegraphics[width=\textwidth]{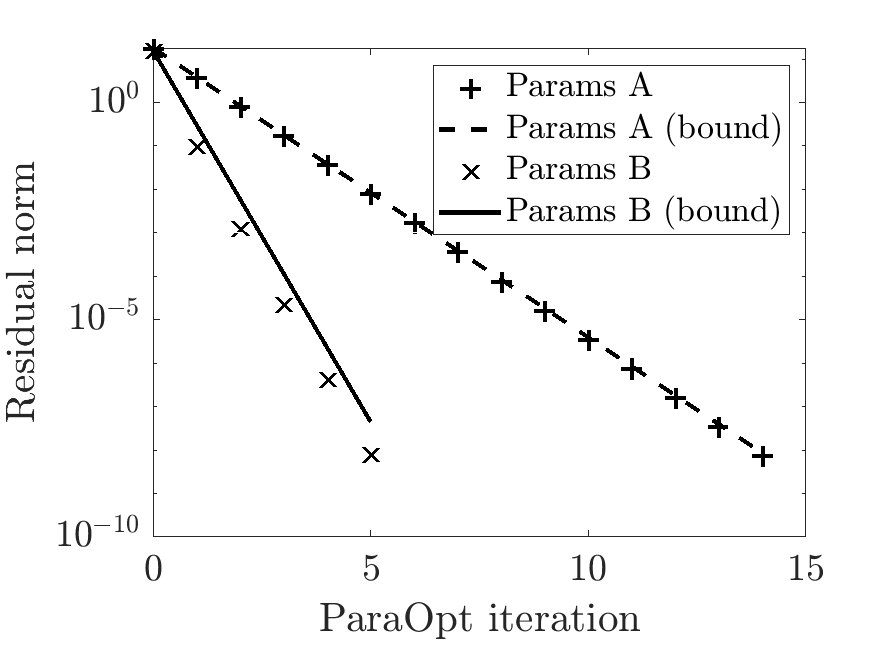}
        \caption{
            True convergence and rate \cref{eq:thm:conv:conv:tr-gen:bound}
            \vspace{-.6cm}
        }
        \label{fig:po-num:scalar:conv-ill:2}
    \end{subfigure}
    \caption{
        Scalar tracking ParaOpt with exact $\mathcal P/\mathcal Q$ and 10-step implicit-Euler $\tilde{\mathcal P}/\tilde{\mathcal Q}$, using the parameters from \cref{sec:num:conv-tr}
        \vspace{-.6cm}
    }
    \label{fig:po-num:scalar:conv-ill}
\end{figure}

\begin{figure}
    \centering
    \begin{subfigure}[c]{.4\textwidth}
        \includegraphics[width=\textwidth]{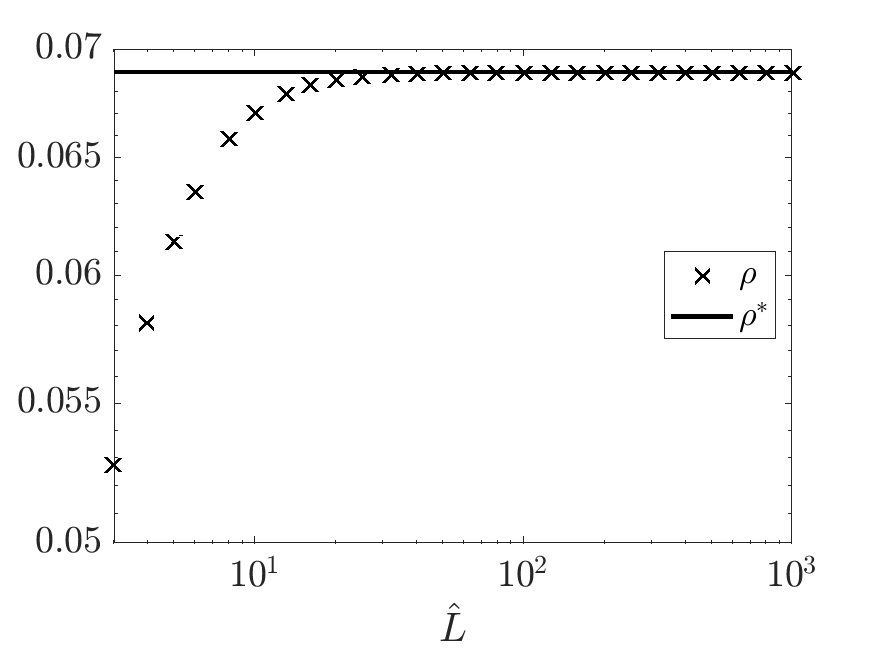}
        \caption{
            Fixed sub-interval size $\DT=1$
            \vspace{-.6cm}
        }
        \label{fig:po-num:scalar:scaling:fixedDT}
    \end{subfigure}
    \hfill
    \begin{subfigure}[c]{.4\textwidth}
        \includegraphics[width=\textwidth]{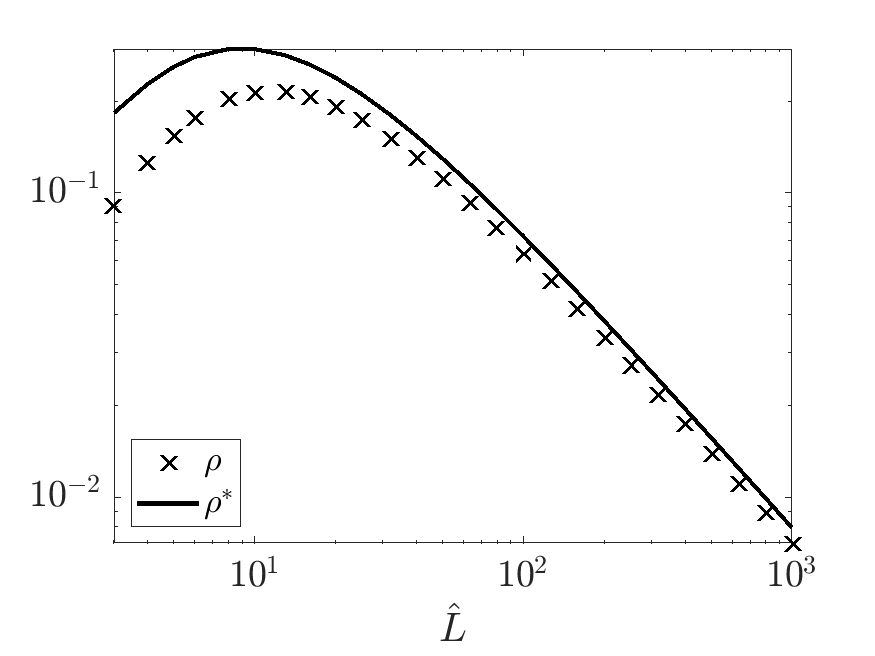}
        \caption{
            Fixed total interval size $T=1$
            \vspace{-.6cm}
        }
        \label{fig:po-num:scalar:scaling:fixedT}
    \end{subfigure}
    \caption{
        Spectral radius $\rho$ and bound $\rho^*$ from \cref{eq:thm:conv:conv:tr-gen:bound} of the tracking ParaOpt iteration matrix, using the parameters from \cref{sec:num:conv-tr}
        \vspace{-.7cm}
    }
    \label{fig:po-num:scalar:scaling}
\end{figure}

%% file: src/num/conv-tc.tex
\begin{figure}
    \centering
    \begin{subfigure}{.4\textwidth}
        \includegraphics[width=\textwidth]{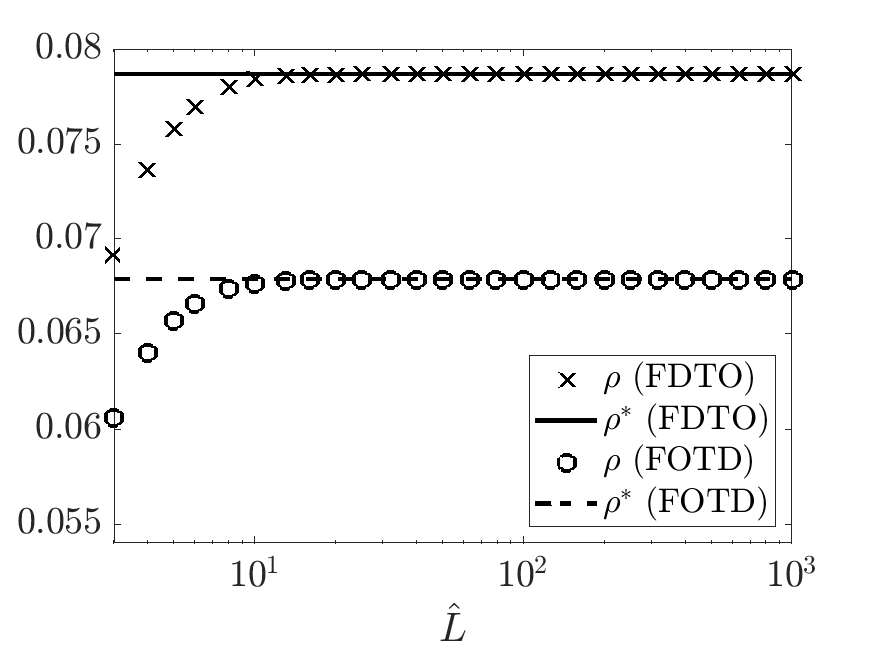}
    \end{subfigure}
    \hfill
    \begin{subfigure}{.4\textwidth}
        \includegraphics[width=\textwidth]{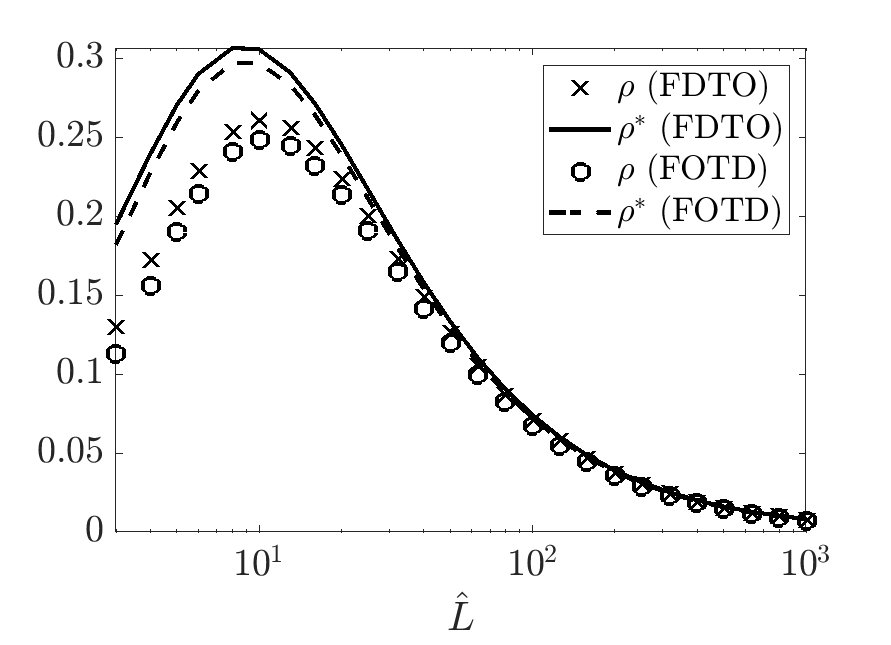}
    \end{subfigure}
    \vfill
    \begin{subfigure}{.4\textwidth}
        \includegraphics[width=\textwidth]{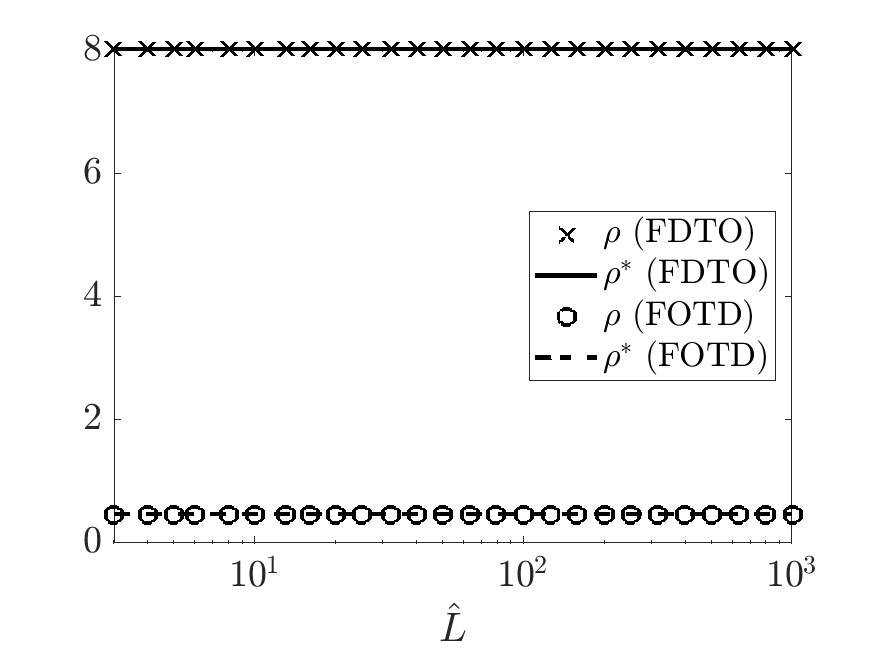}
    \end{subfigure}
    \hfill
    \begin{subfigure}{.4\textwidth}
        \includegraphics[width=\textwidth]{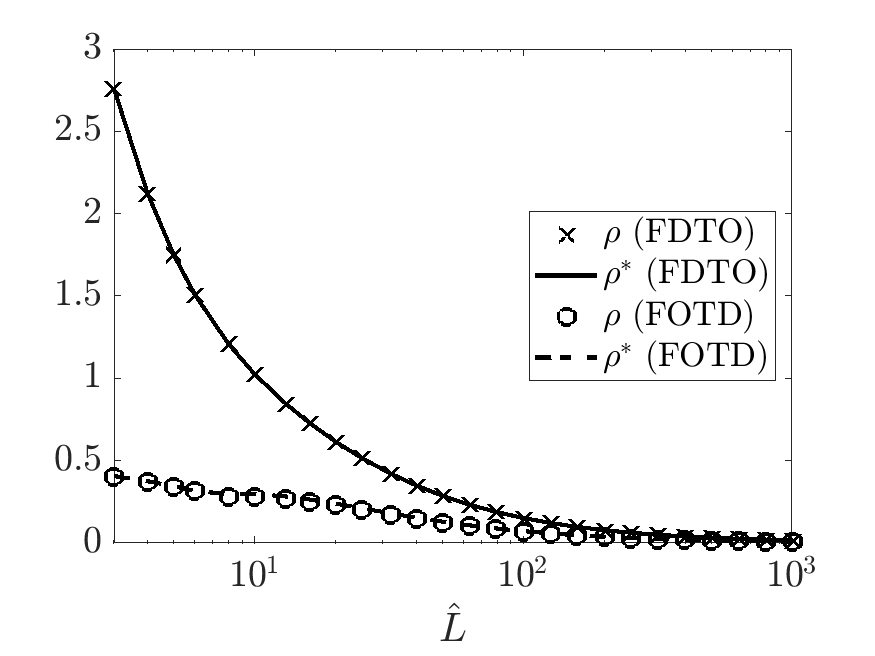}
    \end{subfigure}

    \caption{Equivalents of \cref{fig:po-num:scalar:scaling} for terminal cost with $\gamma=1$ (first row) and $\gamma=\texttt{1e-6}$ (second row)}
    \label{fig:num:conv-tc:scaling}
\end{figure}

The experiments in \cref{sec:num:conv-tr} can be repeated for the terminal-cost case, with the target trajectory $y_\mathrm d(\cdot)=1$ replaced by a target state $y_\mathrm{target} = 1$. In a bid not to qualitatively repeat experiments from \cite{ganderPARAOPTPararealAlgorithm2020a}, we focus on the main difference between our results: the switch from \textsc{fdto} to \textsc{fotd} for the coarse propagator, whose analysis is enabled by our more general convergence bound in \cref{thm:conv:conv:tc-gen}. \Cref{fig:num:conv-tc:scaling} repeats the scaling test from \cref{fig:po-num:scalar:scaling} in the terminal-cost setting. We compare a medium $\gamma=1$ to a very small $\gamma=10^{-6}$ (that is, a large $\widehat\gamma$, meaning control is cheap). As expected from \cref{fig:conv:interp:rhostar:fdtovsfotd}, \textsc{fotd} and \textsc{fdto} differ little in the former case and a lot in the latter. Herein lies the advantage of the \textsc{fotd} implicit-Euler coarse propagator: ParaOpt converges for all linear diffusive problems, not just those with sufficiently small time steps.

%% file: src/num/prec.tex
As discussed in \cref{sec:diag:scale}, the two main processes to study in ParaOpt are the outer inexact-Newton iterations and the inner linear-system solves that constitute the inexact-Newton corrections. The past two subsections have confirmed the scalability of the former -- now, we verify whether the proposed preconditioners succeed in completing the picture by keeping inner iterations constant when the number of parallel time intervals is scaled.

\paragraph{Heat problem}
We study the same heat problem considered in \cite{mezelfparadiag} (and loosely adapted from \cite{emmettEfficientParallelTime2012}), defined on a periodic spatial domain $\Omega = [0, 1]^2$ and given by
\begin{equation}
    \partial_ty = \Delta y + u    \mcom
\end{equation}
where the spatial derivative is discreti\sz{}ed using central differences. We also use \cite{mezelfparadiag}'s initial value, target trajectory and target state
\begin{align*}
    y_\mathrm{init}(x) &= \frac1{12\pi^2\gamma}(1-T)\mathrm{sign}(\sin(2\pi x_1))\sin^2(2\pi x_2)\\
    y_\mathrm{target}(x) &= \sin(2\pi x_1)\sin(2\pi x_2)\\
    y_\mathrm d(t, x) &= \Bigl((12\pi^2+(12\pi^2\gamma)^{-1})(t-T)-(1+((12\pi^2)^2\gamma)^{-1})\Bigr)\sin(2\pi x_1)\sin(2\pi x_2)    \mper
\end{align*}
A non-smooth $y_\mathrm{init}$ is important to make the preconditioning problem sufficiently challenging, as also noticed in \cite{goddardNoteParallelPreconditioning2019,wuParallelInTimeBlockCirculantPreconditioner2020a}. We use the same parameters $\gamma=0.05$ and $T=2$ as \cite{gotschelEfficientParallelinTimeMethod2019a}, and a spatial grid of $\Sidx=8\times8$ points (a small value, to keep the non-preconditioned computation tractable). The fine propagator uses $10$ implicit-Euler steps; the coarse one just $1$. All further results will use a tolerance of $10^{-6}$ for ParaOpt, and will employ a \textsc{gmres} inner solver with a tolerance of $10^{-4}$.
\begin{remark}[ParaOpt's inner \textsc{gmres} tolerance]
    A significant difference between ParaDiag's use as a stand-alone method and its adaptation as a ParaOpt preconditioner is the tolerance to which the system should be solved. Since ParaOpt solves a system with a Jacobian that is already approximate, there is no need to solve the system to full precision. Figure \ref{fig:po-num:prec:gmrestol} shows that, for the heat problem we are about to study with $\widehat\Tidx=10$, $T=2$, $\gamma=0.05$, $\Sidx=8\times8$, a ten-step fine and a one-step coarse implicit-Euler propagator, a tolerance as high as $10^{-3}$ already performs well.
\end{remark}

\begin{figure}[h!]
    \centering
    \begin{subfigure}[c]{.49\textwidth}
        \includegraphics[width=\textwidth]{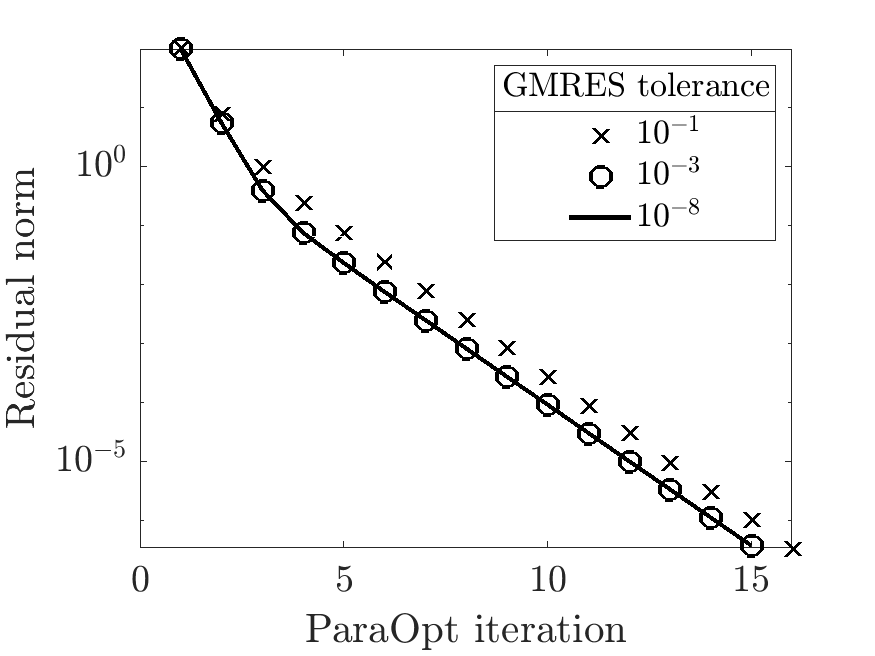}
        \caption{
            Tracking-type objective
        }
        \label{fig:po-num:prec:gmrestol:tr}
    \end{subfigure}
    \hfill
    \begin{subfigure}[c]{.49\textwidth}
        \includegraphics[width=\textwidth]{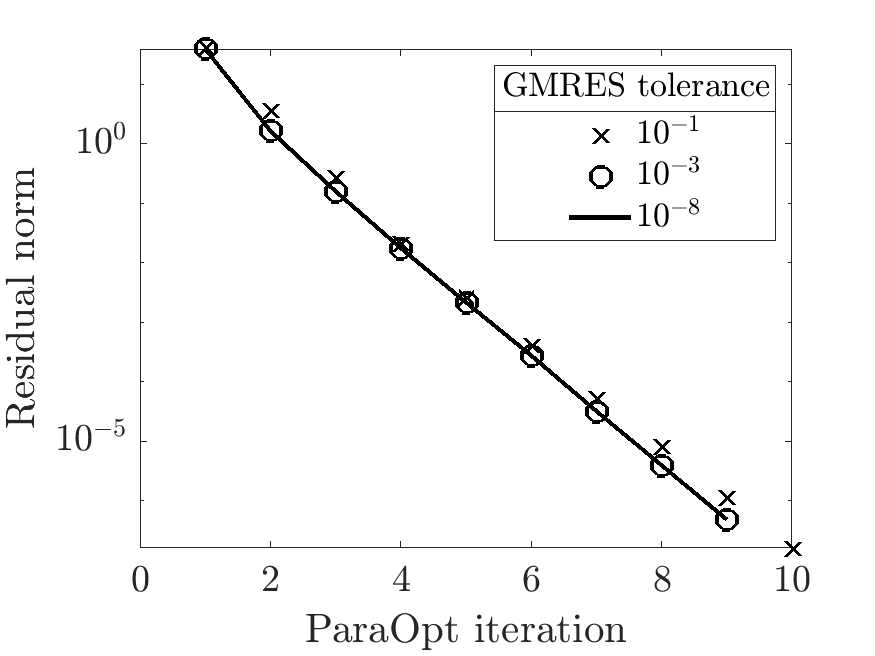}
        \caption{
            Terminal-cost objective
        }
        \label{fig:po-num:prec:gmrestol:tc}
    \end{subfigure}
    \caption{
        Heat problem ParaOpt residual, for different \textsc{gmres} tolerances
    }
    \label{fig:po-num:prec:gmrestol}
\end{figure}

\Cref{fig:po-num:prec:illustr} shows how many \textsc{gmres} iterations are required within each ParaOpt iteration for $\widehat\Tidx\in\{10,100\}$. Using the preconditioner, this number is not only much lower, but also remains virtually constant when $\widehat\Tidx$ changes. The difference preconditioning makes is seen even more strikingly in \cref{fig:po-num:prec:scale}, which sums up the \textsc{gmres} iterations over all ParaOpt iterations. Using our preconditioners, the work per processor remains constant when scaling the processors in tandem with $\widehat\Tidx$ -- in other words, the algorithm is weakly scalable.

\begin{figure}[h!]
    \centering
    \begin{subfigure}[c]{.47\textwidth}
        \includegraphics[width=\textwidth]{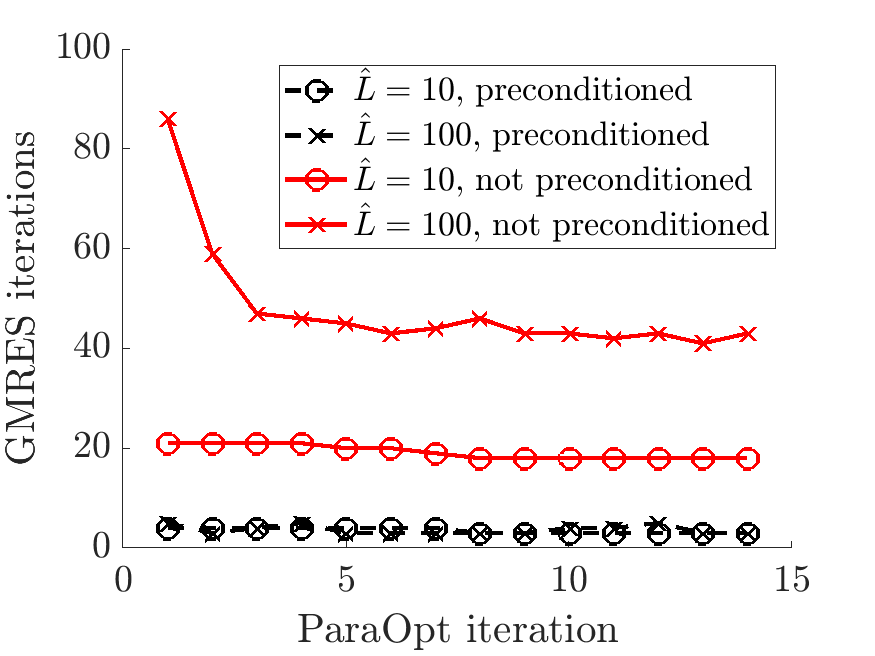}
        \caption{
            Tracking-type objective
        }
        \label{fig:po-num:prec:illustr:tr}
    \end{subfigure}
    \hfill
    \begin{subfigure}[c]{.47\textwidth}
        \includegraphics[width=\textwidth]{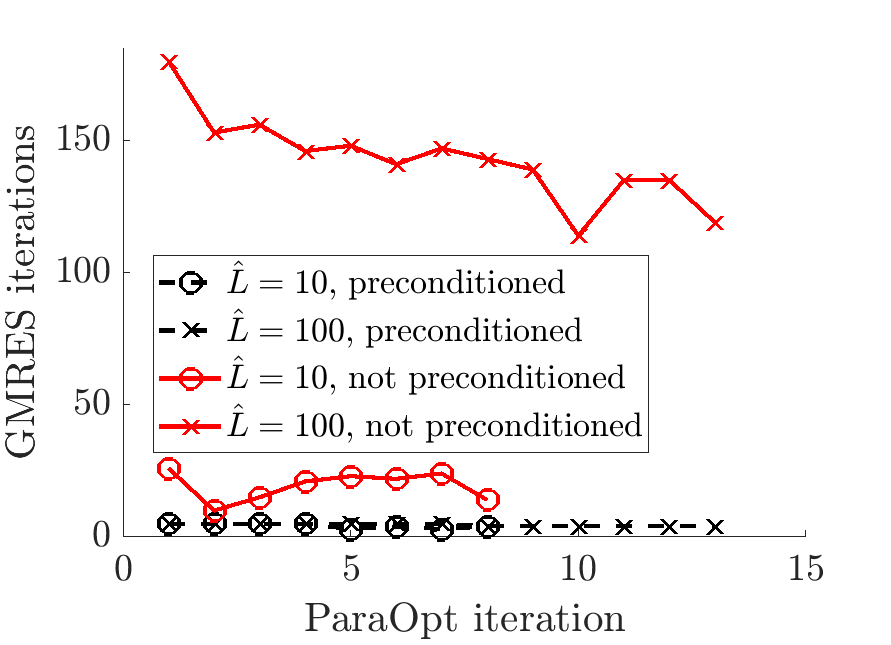}
        \caption{
            Terminal-cost objective
        }
        \label{fig:po-num:prec:illustr:tc}
    \end{subfigure}
    \vspace{-.2cm}
    \caption{
        (Un)preconditioned \textsc{gmres} iteration counts for the heat problem
    }
    \label{fig:po-num:prec:illustr}
\end{figure}
\begin{figure}[h!]
    \centering
    \begin{subfigure}[c]{.47\textwidth}
        \includegraphics[width=\textwidth]{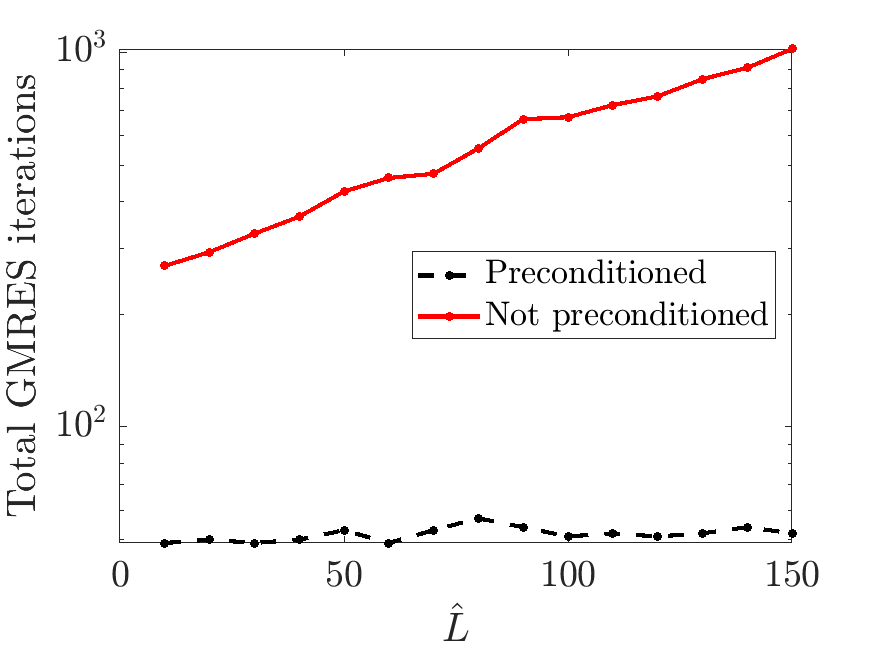}
        \caption{
            Tracking-type objective
        }
        \label{fig:po-num:prec:scale:tr}
    \end{subfigure}
    \hfill
    \begin{subfigure}[c]{.47\textwidth}
        \includegraphics[width=\textwidth]{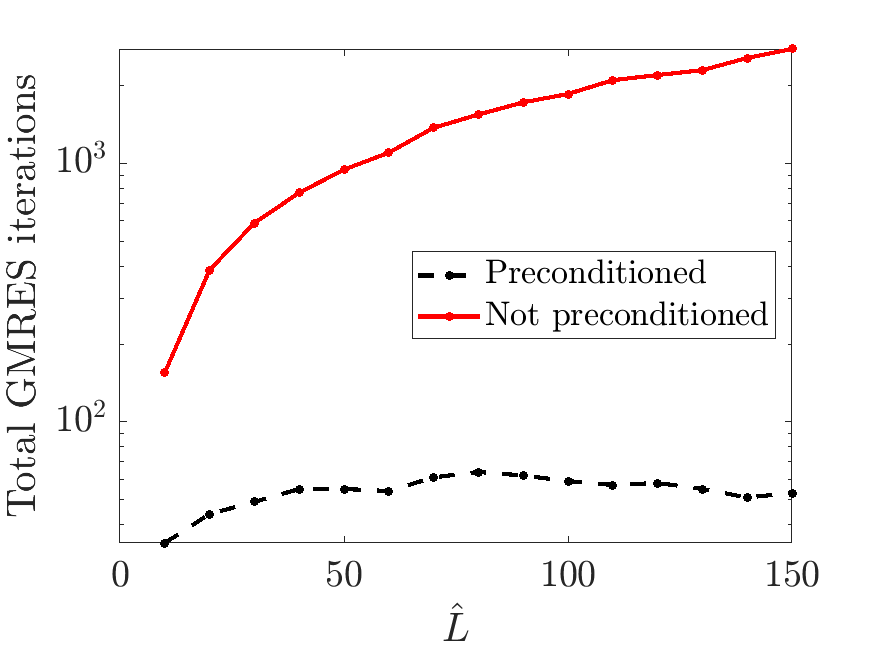}
        \caption{
            Terminal-cost objective
        }
        \label{fig:po-num:prec:scale:tc}
    \end{subfigure}
    \vspace{-.2cm}
    \caption{
        (Un)preconditioned total \textsc{gmres} iteration counts for the heat problem
    }
    \label{fig:po-num:prec:scale}
\end{figure}

\clearpage
\paragraph{Advection-diffusion problem}
In addition to the diffusion problem which follows the theory in \cref{sec:conv} and \cite{mezelfparadiag}, we study an advection-diffusion problem. Its $K$ matrix is not symmetric, such that our theory does not apply -- however, the preconditioners from \cref{sec:diag} can still be used. Consider the equation
\begin{equation}
    \partial_ty = \frac{\Delta y}{10} - \partial_{x_1}y - \partial_{x_2}y + u    \mper
\end{equation}
which retains some diffusion (otherwise, ParaOpt itself might have poor convergence, unrelated to the preconditioners) but adds an advection term. We use the same $y_\mathrm d$, $y_\mathrm{target}$, and $y_\mathrm{init}$ as for the heat problem and again discreti\sz{}e spatial derivatives using central differences. \Cref{fig:po-num:prec:da-illustr,fig:po-num:prec:da-scale} are the advection-diffusion counterparts to \cref{fig:po-num:prec:illustr,fig:po-num:prec:scale}, to which they are qualitatively very similar. We can conclude that the proposed preconditioner performs well, even outside the regime where it is fully understood. A similar conclusion was drawn in the ParaDiag context \cite{mezelfparadiag}.

\begin{figure}
    \centering
    \begin{subfigure}[c]{.47\textwidth}
        \includegraphics[width=\textwidth]{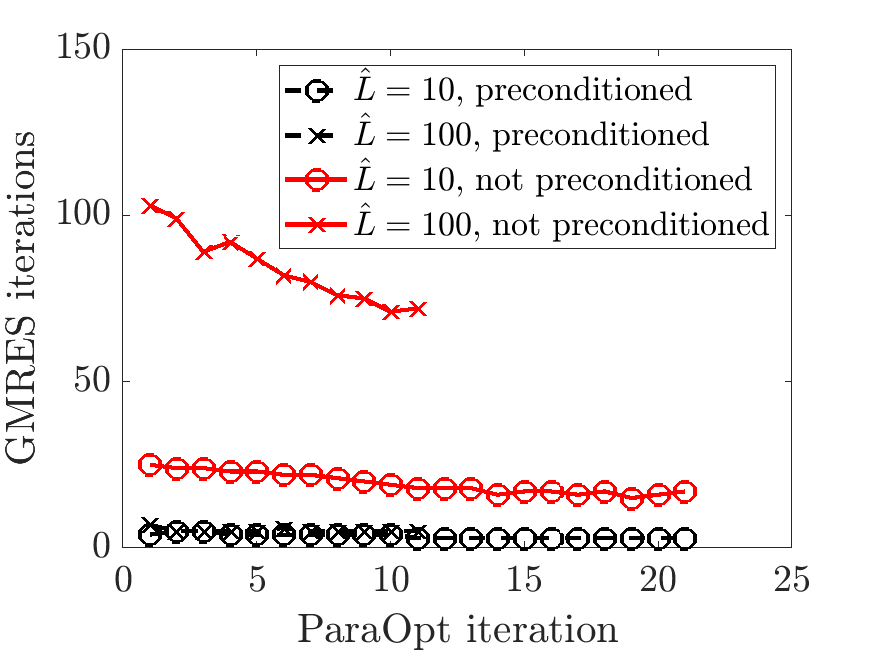}
        \caption{
            Tracking-type objective
        }
        \label{fig:po-num:prec:illustr:da-tr}
    \end{subfigure}
    \hfill
    \begin{subfigure}[c]{.47\textwidth}
        \includegraphics[width=\textwidth]{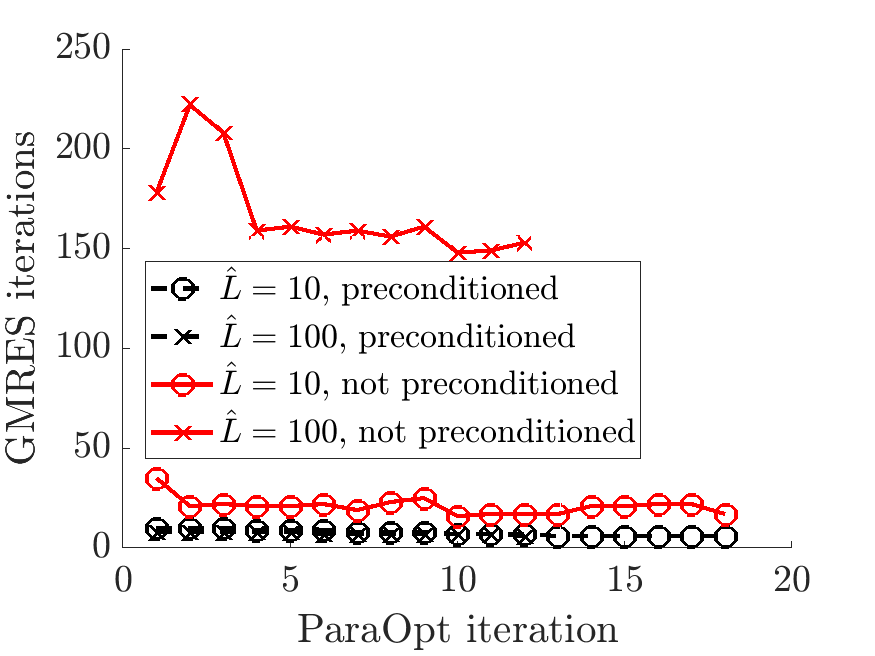}
        \caption{
            Terminal-cost objective
        }
        \label{fig:po-num:prec:illustr:da-tc}
    \end{subfigure}
    \vspace{-.2cm}
    \caption{
        (Un)preconditioned \textsc{gmres} iteration counts for the advection-diffusion problem
    }
    \label{fig:po-num:prec:da-illustr}
\end{figure}
\begin{figure}
    \centering
    \begin{subfigure}[c]{.47\textwidth}
        \includegraphics[width=\textwidth]{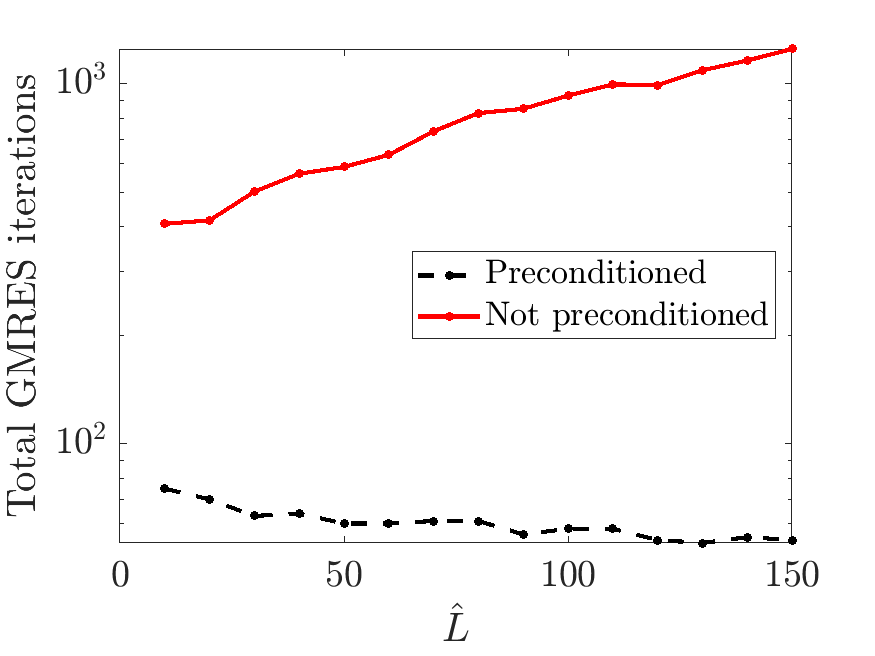}
        \caption{
            Tracking-type objective
        }
        \label{fig:po-num:prec:scale:da-tr}
    \end{subfigure}
    \hfill
    \begin{subfigure}[c]{.47\textwidth}
        \includegraphics[width=\textwidth]{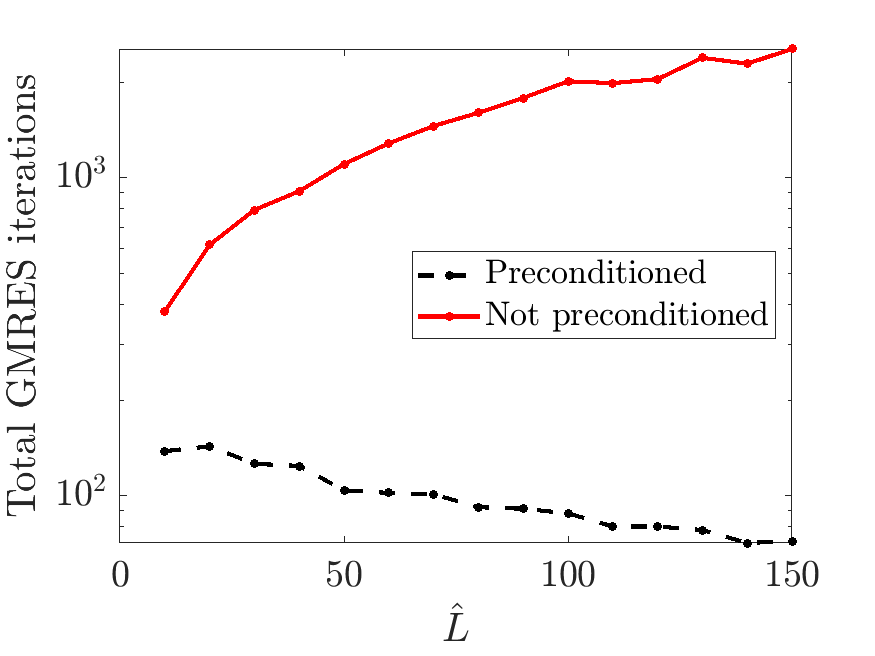}
        \caption{
            Terminal-cost objective
        }
        \label{fig:po-num:prec:scale:da-tc}
    \end{subfigure}
    \vspace{-.2cm}
    \caption{
        (Un)preconditioned total \textsc{gmres} iteration counts for the advection-diffusion problem
    }
    \label{fig:po-num:prec:da-scale}
\end{figure}

%% file: src/concl/concl.tex
The main focus of this paper is on linear diffusive problems. In that setting, and when \cref{ass:conv:setting:1,ass:conv:setting:2,ass:conv:setting:3} are satisfied, we adapted ParaOpt to work for both tracking and terminal-cost objectives, additionally proposing generic convergence bounds and a preconditioner that ensures good weak scalability. We summari\sz{}e our main contributions in the linear diffusive setting. \begin{itemize}
    \item We formulated an extension of ParaOpt to the setting of tracking problems.
    \item We proved a generali\sz{}ed convergence bound for terminal-cost ParaOpt that is generic in the propagators used (as long as they satisfy \cref{ass:conv:setting:1,ass:conv:setting:2,ass:conv:setting:3}). Numerical results confirmed the bound's validity and showed that it is even sharper than the state-of-the-art result \cite{ganderPARAOPTPararealAlgorithm2020a}.
    \item We proved a similar bound for the tracking setting and proved that the case of an exact fine and a one-step implicit-Euler coarse propagator guarantees ParaOpt does not diverge.
    \item Thanks to the generic convergence bound, we were able to study a different type of implicit-Euler coarse propagator for terminal-cost problems and found that it improves ParaOpt's convergence.
    \item We proposed diagonali\sz{}ation-based preconditioners to improve the scaling of ParaOpt. Analytic results from the ParaDiag literature and numerical tests confirm that ParaOpt is now a weakly scalable algorithm.
\end{itemize}

For problems outside the linear diffusive category, ParaOpt is still a very promising method \cite{ganderPARAOPTPararealAlgorithm2020a}, although poorly understood. Our preconditioners apply as long as the coarse propagator is affine in its arguments, even for non-diffusive and non-symmetric problems. In addition, we are confident that these preconditioners can be extended to non-linear coarse propagators, although we leave this for future work. Recall that our application of diagonali\sz{}ation-based preconditioners to ParaOpt was inspired by proceedings in \cite{wuParallelCoarseGrid2018a} for the \textsc{ivp} Parareal algorithm. There, non-linear Parareal is supported by using non-linear variants of ParaDiag. For \textsc{ivp}s, these non-linear ParaDiag methods exist \cite{ganderTimeParallelizationNonlinear2017a,liuFastBlockAcirculant2020a}. Developing similar techniques for optimi\sz{}ation ParaDiag would open the way for preconditioning non-linear ParaOpt similarly to what we proposed.

Other interesting future work could consist of applying ParaOpt to realistic applications, testing the proposed preconditioners on more complex (but still linear) problems. In addition, with the generic bounds in \cref{sec:conv:conv}, the search for coarse propagators with better convergence properties than implicit Euler is wide open: one simply needs to calculate a propagator's $\tilde\varphi$ and $\tilde\psi$ (see \cref{sec:apdx-po-prop:phipsi}) to evaluate its performance.

We conclude by noting that ParaOpt does not necessarily need a parallel preconditioner; the fine propagations can already be paralleli\sz{}ed, so as long as the coarse-grid correction is cheap, it need not be parallel for the method itself to achieve speed-ups (in fact, the base version of Parareal has sequential coarse-grid correction). It would therefore be feasible for sequential preconditioners to match or surpass the performance of the ones proposed here, if they cause the system to be solved in fewer iterations.

%% file: src/apdx-po-prop/intro.tex
This appendix contains calculations for various properties of the ParaOpt propagators in \cref{sec:conv}.

%% file: src/apdx-po-prop/phipsi.tex

Consider a time interval $[T_{\tidx-1}, T_\tidx]$ with $T_\tidx-T_{\tidx-1}=\DT$. Then the tracking propagators $\mathcal P(\mv y_{\tidx-1},\mv{\widehat\ad}_\tidx)$ and $\mathcal Q(\mv y_{\tidx-1},\mv{\widehat\ad}_\tidx)$ for a linear diffusive system approximately solve the system\begin{subequations} \label{eq:apdx-po-prop:phipsi:subsys}
\begin{align}
    \mv y'(t) &= -K \mv y(t) - \mv{\widehat\ad}(t)/\sqrt\gamma, &\quad \mv y(T_{\tidx-1}) &= \mv y_{\tidx-1}    \mcom\\
    \mv{\widehat\ad}'(t) &= (\mv{y_\mathrm d}(t) - \mv y(t))/\sqrt\gamma + K\mv{\widehat\ad}(t), &\quad \mv{\widehat\ad}(T_\tidx) &= \mv{\widehat\ad}_\tidx    \mper
\end{align}
\end{subequations}
\Cref{ass:conv:setting:1,ass:conv:setting:2} can now be checked, either by reasoning about the discreti\sz{}ation scheme or by fully calculating $\Phi$ and $\Psi$. In the former case, the eigenvalues $\varphi$ and $\psi$ can then be found by considering the scalar problem\begin{subequations}\label{eq:apdx-po-prop:phipsi:subsys-scalar}
\begin{align}
    y'(t) &= -\sigma y(t) -{\widehat\ad}(t)/\sqrt\gamma, &\quad y(T_{\tidx-1}) &= y_{\tidx-1}\\
    {\widehat\ad}'(t) &= ({y_\mathrm d}(t) - y(t))/\sqrt\gamma + \sigma{\widehat\ad}(t), &\quad{\widehat\ad}(T_\tidx) &= {\widehat\ad}_\tidx    \mper
\end{align}
\end{subequations}
A discreti\sz{}ation of \cref{eq:apdx-po-prop:phipsi:subsys-scalar} and one of the expressions 
\begin{subequations}
\begin{align}
    y(T_\tidx) &\approx \mathcal P(y_{\tidx-1},\widehat\ad_\tidx) = \varphi y_{\tidx-1} - \psi\widehat\ad_\tidx + b_{\mathcal P,\tidx}    \mcom\\
    \widehat\ad(T_{\tidx-1}) &\approx \mathcal Q(y_{\tidx-1},\widehat\ad_\tidx) = \psi y_{\tidx-1} + \varphi\widehat\ad_\tidx + b_{\mathcal Q,\tidx}
\end{align}
\end{subequations} together yield $\varphi$ and $\psi$. An analogous procedure can be employed for terminal cost.

%% file: src/apdx-po-prop/prop-tr-ie.tex
It is clear that \cref{ass:conv:setting:1,ass:conv:setting:2} hold for an \textsc{fotd} implicit-Euler discreti\sz{}ation of \cref{eq:apdx-po-prop:phipsi:subsys}, which discreti\sz{}es the state equation forward in time with implicit Euler and the adjoint equation backward. Indeed, affinity and simultaneous diagonali\sz{}ability follow from the fact that the discreti\sz{}ation can be written as a large linear system using only $K$ and scaled identity matrices, while the rest of \cref{ass:conv:setting:2} follows from the symmetry of \cref{eq:apdx-po-prop:phipsi:subsys} and the discreti\sz{}ation.

When using implicit Euler with $J$ steps to discreti\sz{}e the system \cref{eq:apdx-po-prop:phipsi:subsys-scalar}, let us -- with a slight abuse of notation -- introduce new indices $j$ on $y$ and $\ad$. The starting index $(\tidx-1)$ becomes $j=0$, $\tidx$ becomes $j=J$, and intermediate $j$ values are used for the finer grid of implicit Euler. The time step is $\tau = \DT/J$. Define $\zeta\coloneqq(1+\sigma\tau)$, such that the implicit-Euler discreti\sz{}ation reads (for $j = 1, \ldots, J$)
\begin{subequations} \label{eq:apdx-po-prop:prop-tr-ie:discr}
\begin{align}
    y_j - \zeta^{-1}y_{j-1} + \tau\zeta^{-1}\widehat\ad_j/\sqrt\gamma &= 0    \mcom\\
    \widehat\ad_{j-1} - \zeta^{-1}\widehat\ad_j - \tau\zeta^{-1}y_{j-1}/\sqrt\gamma &= -\tau\zeta^{-1}y_{\mathrm d,j-1}/\sqrt\gamma    \mper
\end{align}
\end{subequations}

Let us now define $\varphi_\tau^{(j)}$ and $\psi_\tau^{(j)}$ as the $\varphi$ and $\psi$ for $j$ length-$\tau$ steps -- then the variables we are looking for are $\varphi_\tau^{(J)}$ and $\psi_\tau^{(J)}$. For any $j$, we can write
\begin{equation} \label{eq:apdx-po-prop:prop-tr-ie:phipsi-j}
    -\varphi_\tau^{(j)}y_0 + y_j + \psi_\tau^{(j)}\widehat\ad_j = b_j \quad \text{for some $b_j$}   \mper
\end{equation}
This gives a base case $\varphi_\tau^{(0)} = 1$ and $\psi_\tau^{(0)} = 0$. Now suppose that $\varphi_\tau^{(j)}$ and $\psi_\tau^{(j)}$ are known for some $j$. Combining \cref{eq:apdx-po-prop:prop-tr-ie:phipsi-j} (for $j+1$) with the discreti\sz{}ation \cref{eq:apdx-po-prop:prop-tr-ie:discr} gives
\begin{equation*}
\begin{aligned}
    &&-\varphi_\tau^{(j+1)}y_0 + y_{j+1} + \psi_\tau^{(j+1)}\widehat\ad_{j+1} &= b_{j+1}\\
    &\Leftrightarrow & -\varphi_\tau^{(j+1)}y_0 + (\zeta^{-1}y_{j}-\widehat\gamma_\tau\zeta^{-1}\widehat\ad_{j+1}) + \psi_\tau^{(j+1)}\widehat\ad_{j+1} &= b_{j+1}\\
    &\Leftrightarrow & -\varphi_\tau^{(j+1)}y_0 + \zeta^{-1}y_{j} + (\psi_\tau^{(j+1)}-\widehat\gamma_\tau\zeta^{-1})(\zeta\widehat\ad_{j}-\widehat\gamma_\tau y_{j}) &= \widehat b_{j+1}\\
    &\Leftrightarrow & -\varphi_\tau^{(j+1)}y_0 + (\zeta^{-1}-\widehat\gamma_\tau(\psi_\tau^{(j+1)}-\widehat\gamma_\tau\zeta^{-1}))y_{j} + (\psi_\tau^{(j+1)}\zeta - \widehat\gamma_\tau)\widehat\ad_{j} &= \widehat b_{j+1}\\
\end{aligned}
\end{equation*}
for some $\widehat b_j$. Comparing this to \cref{eq:apdx-po-prop:prop-tr-ie:phipsi-j}, a few algebraic manipulations yield \cref{eq:lmm:conv:setting:prop-tr-ie:2}.

\clearpage
We should now check that \cref{ass:conv:setting:3} holds if $\sigma>0$. The expression for $\psi$ is clearly positive. We further have $\varphi_\tau^{(0)}=1$ and, for each subsequent $j$, $\varphi_\tau^{(j)}$ is equal to $\varphi_\tau^{(j-1)}$ multiplied by a factor $r^{(j)}\coloneqq(\zeta^{-1}-\widehat\gamma(\psi_\tau^{(j)}-\widehat\gamma\zeta^{-1}))$. We claim that $0<r^{(j)}<1$ for $j\ge1$, meaning that $0<\varphi_\tau^{(j)}<1$ for $j\ge1$. Indeed,
\begin{equation*}
\begin{aligned}
               && 0<r^{(j)}<1 \Leftrightarrow 0&<1-\zeta\widehat\gamma(\psi_\tau^{(j)}-\widehat\gamma\zeta^{-1})<\zeta\\
&\Leftrightarrow& (1+\widehat\gamma^2)/(\zeta\widehat\gamma) &> \psi_\tau^{(j)} > (1+\widehat\gamma^2-\zeta)/(\zeta\widehat\gamma)\\
&\Leftrightarrow& \frac{1+\widehat\gamma^2}{\zeta\widehat\gamma} &> \frac{\widehat\gamma+\zeta^{-1}(1+\widehat\gamma^2)\psi_\tau^{(j-1)}}{\zeta + \widehat\gamma\psi_\tau^{(j-1)}} > \frac{1+\widehat\gamma^2-\zeta}{\zeta\widehat\gamma}\\
&\Leftrightarrow& \frac1{\zeta\widehat\gamma} &> \frac{\psi_\tau^{(j-1)}}{\zeta^2+\zeta\widehat\gamma\psi_\tau^{(j-1)}} > \frac{1-\zeta}{\zeta\widehat\gamma}    \mper
\end{aligned}
\end{equation*}
These last inequalities hold; to see this, multiply the numerator and denominator of the leftmost expression by $\psi_\tau^{(j-1)}$ and note that the rightmost expression is negative.

%% file: src/apdx-po-prop/prop-tr-ex.tex
For an exact solution to \cref{eq:apdx-po-prop:phipsi:subsys}, note that
\begin{equation} \label{eq:apdx-po-prop:prop-tr-ex:mexp}
    \begin{litmat}
        \mv y_\tidx\\\mv{\widehat\ad}_\tidx
    \end{litmat} = \exp\left(
        \DT\begin{litmat}
            -K & -I/\sqrt\gamma\\
            -I/\sqrt\gamma & K
        \end{litmat}
    \right)\begin{litmat}
        \mv y_{\tidx-1}\\\mv{\widehat\ad}_{\tidx-1}
    \end{litmat} + \begin{litmat}
        \mv v\\\mv w
    \end{litmat}
\end{equation}
for some $\mv v$ and $\mv w$ (both dependent on $\mv{y_\mathrm d}$) that are of little importance here. \Cref{ass:conv:setting:1,ass:conv:setting:2} follow easily from this expression. We can then consider the scalar case (by replacing $K$ with $\sigma$ and $I$ with $1$) and denote the matrix exponential in \cref{eq:apdx-po-prop:prop-tr-ex:mexp} as $E \eqqcolon \bigl[\begin{smallmatrix}
    \cdot & \cdot\\c & d\\
\end{smallmatrix}\bigr]$. We obtain that $\widehat\ad_\tidx = cy_{\tidx-1}+d\widehat\ad_{\tidx-1}$ and, thus,
\begin{equation} \label{eq:apdx-po-prop:prop-tr-ex:phipsi}
    \varphi = d^{-1} \quad \text{and} \quad \psi = -d^{-1}c    \mper
\end{equation}

To calculate $E$, finding $c$ and $d$, denote $M \coloneqq \DT\bigl[\begin{smallmatrix}
        -\sigma & -1/\sqrt\gamma\\
        -1/\sqrt\gamma & \sigma
    \end{smallmatrix}\bigr] = \bigl[\begin{smallmatrix}
        -\widehat\sigma & -\widehat\gamma\\
        -\widehat\gamma & \widehat\sigma
    \end{smallmatrix}\bigr]$.
We define $s\coloneqq\sqrt{\widehat\sigma^2+\widehat\gamma^2}$. It can be checked that $M=V\Sigma V^{-1}$ with
\begin{equation}
    \Sigma = \begin{litmat}
        s &\\&-s
    \end{litmat}, \quad V = \begin{litmat}
        \frac{\widehat\sigma - s}{\widehat\gamma} & \frac{\widehat\sigma + s}{\widehat\gamma}\\
        1 & 1
    \end{litmat}, \quad \text{and} \quad V^{-1} = \begin{litmat}
        \frac{-\widehat\gamma}{2s} & \frac{\widehat\sigma+s}{2s}\\
        \frac{\widehat\gamma}{2s} & \frac{-\widehat\sigma+s}{2s}
    \end{litmat}    \mper
\end{equation}
Then $\exp(M) = V\exp(\Sigma)V^{-1}$, where the exponential of a diagonal matrix can be distributed to its entries. We obtain
\begin{equation}
\begin{aligned}
    \exp(M) &= \begin{litmat}
        \frac{\widehat\sigma-s}{\widehat\gamma} & \frac{\widehat\sigma+s}{\widehat\gamma}\\
        1 & 1\\
    \end{litmat}\begin{litmat}
        \E^s\\&\E^{-s}
    \end{litmat}\begin{litmat}
        \frac{-\widehat\gamma}{2s} & \frac{\widehat\sigma+s}{2s}\\
        \frac{\widehat\gamma}{2s} & \frac{-\widehat\sigma+s}{2s}\\
    \end{litmat}\\
    &= \begin{litmat}
        -\E^s\frac{\widehat\gamma}{2s}\frac{\widehat\sigma-s}{\widehat\gamma} + \E^{-s}\frac{\widehat\gamma}{2s}\frac{\widehat\sigma+s}{\widehat\gamma} & \E^s\frac{\widehat\sigma+s}{2s}\frac{\widehat\sigma-s}{\widehat\gamma} + \E^{-s}\frac{-\widehat\sigma+s}{2s}\frac{\widehat\sigma+s}{\widehat\gamma}\\
        -\E^s\frac{\widehat\gamma}{2s} + \E^{-s}\frac{\widehat\gamma}{2s} & \E^s\frac{\widehat\sigma+s}{2s} + \E^{-s}\frac{-\widehat\sigma+s}{2s}\\
    \end{litmat}\\
    &= \begin{litmat}
        \cosh s - \widehat\sigma\frac{\sinh s}s & -\widehat\gamma\frac{\sinh s}s\\
        -\widehat\gamma\frac{\sinh s}s & \cosh s + \widehat\sigma\frac{\sinh s}s\\
    \end{litmat}    \mcom
\end{aligned}
\end{equation}
from which follow
\begin{equation}
    c = -\widehat\gamma\frac{\sinh(\sqrt{\widehat\gamma^2 + \widehat\sigma^2})}{\sqrt{\widehat\gamma^2 + \widehat\sigma^2}} \quad \text{and} \quad d = \cosh(\sqrt{\widehat\gamma^2 + \widehat\sigma^2}) + \widehat\sigma\frac{\sinh(\sqrt{\widehat\gamma^2 + \widehat\sigma^2})}{\sqrt{\widehat\gamma^2 + \widehat\sigma^2}}    \mper
\end{equation}
\Cref{ass:conv:setting:3} clearly holds: $\psi$ is positive and $d>1$ (since it is the sum of a hyperbolic cosine and a positive number), meaning $0<\varphi=d^{-1}<1$.

%% file: src/apdx-po-prop/special.tex
We will need to prove that the right-hand side of \cref{eq:thm:conv:conv:tr-gen:bound} is smaller than $1$ when $\varphi\eqqcolon\varphi_\ex$ and $\psi\eqqcolon\psi_\ex$ are given by \cref{eq:apdx-po-prop:prop-tr-ex:phipsi}, and $\tilde\varphi\eqqcolon\varphi_\DT=\varphi_\DT^{(1)}$ and $\tilde\psi\eqqcolon\psi_\DT=\psi_\DT^{(1)}$ are given by \cref{eq:apdx-po-prop:prop-tr-ie:phipsi-j}. We introduce some auxiliary variables to simplify working with the exact propagators:
\begin{equation*}
    s \coloneqq \sqrt{\widehat\sigma^2 + \widehat\gamma^2}, \quad a\coloneqq\cosh(s), \quad \text{and} \quad b\coloneqq\frac{\sinh(s)}s
\end{equation*}
such that $d=a+\widehat\sigma b$ and $\abs c = -c = \widehat\gamma b$. Since $\widehat\sigma>0$, it is trivial to see that $a,b,d>1$. Furthermore, $a>b$, as becomes clear from their respective Maclaurin series
\begin{equation} \label{eq:apds-po-prop:special:mclaurin}
    \cosh(s) = 1 + \frac{s^2}{2!} + \frac{s^4}{4!} + \cdots \quad \text{and} \quad \frac{\sinh(s)}s = 1 + \frac{s^2}{3!} + \frac{s^4}{5!} + \cdots    \mper
\end{equation}

To prove the bound in \cref{ex:conv:conv:special}, we write
\begin{align*}
    &&\rho^* &\le 1\\
&\Leftrightarrow&    (\varphi_\DT-\varphi_\ex)^2 + (\psi_\DT-\psi_\ex)^2 &\le (1-\varphi_\DT)^2 + \psi_\DT^2\\
&\Leftrightarrow&    2\varphi_\DT\varphi_\ex - \varphi_\ex^2 + 2\psi_\DT\psi - \psi_\ex^2 + 1 - 2\varphi_\DT &\ge 0\\
&\Leftrightarrow&    2\varphi_\ex - (1+\widehat\sigma)\varphi_\ex^2 + 2\widehat\gamma\psi_\ex - (1+\widehat\sigma)\psi_\ex^2 + (1 + \widehat\sigma) - 2 &\ge 0\\
&\Leftrightarrow&    2d(1+\widehat\gamma\abs c) - d^2(1-\widehat\sigma) - (1+\widehat\sigma)(1+\abs c^2) &\ge 0    \mper
\end{align*}
This expression contains one term that is twice a positive quantity, and two that might be negative. It holds if a single instance of the first term plus either one of the others is positive. The first of these conditions is
\begin{equation} \label{eq:apdx-po-prop:special:ineq1-1}
\begin{aligned}
    &&d(1+\widehat\gamma\abs c) \ge d^2(1-\widehat\sigma) \Leftrightarrow 1 + \widehat\gamma\abs c &\ge d(1-\widehat\sigma)\\
&\Leftrightarrow&    1 + \widehat\gamma^2b &\ge (a+\widehat\sigma y)(1-\widehat\sigma)\\
&\Leftarrow&    1 + \widehat\gamma^2b &\ge a(1-\widehat\sigma^2)    \mper\\
&\Leftrightarrow&    1 + \widehat\gamma^2b - a(1-\widehat\sigma^2) &\ge 0    \mcom
\end{aligned}
\end{equation}
where the unidirectional implication holds since $a>b$. To show that \cref{eq:apdx-po-prop:special:ineq1-1}'s last inequality holds, consider its left-hand side's derivative with respect to $\widehat\gamma^2$:
\begin{equation}
\begin{aligned}
    &\frac\dif{\dif\widehat\gamma^2}(1+\widehat\gamma^2b-a(1-\widehat\sigma^2))\\ &= b + \widehat\gamma^2\frac\dif{\dif\widehat\gamma^2}b - (1-\widehat\sigma^2)\frac\dif{\dif\widehat\gamma^2}a
    = b + \widehat\gamma^2(\frac{\cosh s}{2s^2}-\frac{\sinh s}{2s^3}) - (1-\widehat\sigma^2)\frac{\sinh s}{2s}\\
    &= (1-\frac12+\frac{\widehat\sigma^2}2-\frac{\widehat\gamma^2}{2s^2})b + \frac{\widehat\gamma^2}{2s^2}a
    = \frac{\widehat\sigma^2+\widehat\sigma^2\widehat\gamma^2+\widehat\sigma^4}{2s^2}b + \frac{\widehat\gamma^2}{2s^2}a    \mper
\end{aligned}
\end{equation}
This derivative is always positive. In other words, if \cref{eq:apdx-po-prop:special:ineq1-1}'s last inequality holds for $\widehat\gamma\rightarrow0$, it holds for all $\widehat\gamma$. In this limit, \cref{eq:apdx-po-prop:special:ineq1-1} becomes
\begin{equation}
    1 \ge \cosh(\widehat\sigma)(1-\widehat\sigma^2)    \mcom
\end{equation}
which holds for any $\widehat\sigma$ (multiply the first Maclaurin series in \cref{eq:apds-po-prop:special:mclaurin} in $\widehat\sigma$ by $1-\widehat\sigma^2$).

The second inequality from earlier takes less effort:
\begin{equation*}
\begin{aligned}
    &&d(1+\widehat\gamma \abs c) &\ge (1+\widehat\sigma)(1+\abs c^2)
\Leftrightarrow    (a+\widehat\sigma b)(1+\widehat\gamma^2b) \ge (1+\widehat\sigma)(1+\widehat\gamma^2b^2) \Leftarrow\\
  &&  b(1+\widehat\sigma)(1+\widehat\gamma^2b) &\ge (1+\widehat\sigma)(1+\widehat\gamma^2b^2)
\Leftrightarrow    (1+\widehat\sigma)(y+\widehat\gamma^2b^2) \ge (1+\widehat\sigma)(1+\widehat\gamma^2b^2)    \mcom
\end{aligned}
\end{equation*}
where, again, $a>b$ justifies the unidirectional implication. Since $b>1$, this inequality holds and, together with the previous inequality, the bound is now proven.

%% file: src/ack.tex
The authors thank Ignace Bossuyt, Toon Ingelaere, and Vince Maes for their reviews and helpful comments, and Carlos Fonseca for pointing to \cite{betterThanYueh} as the original reference to use in \cref{lmm:paraopt:M2eig}'s proof. This project received funding from the European High-Performance Computing Joint Undertaking (JU) under grant agreement No.\ 955701. The JU receives support from the EU's Horizon 2020 programme. Karl Meerbergen's work is supported by the Research Foundation Flanders grants G0B7818N and G088622N, and by the KU Leuven Research Council.